\documentclass[11pt,leqno]{article}
\usepackage{amsmath, amscd, amsthm, amssymb, graphics, xypic, mathrsfs, setspace, fancyhdr, bm, pdfsync, enumitem, mathptmx}
\usepackage[usenames, dvipsnames, svgnames, table]{xcolor}
\usepackage[letterpaper,top=1.25in, bottom=1.25in, left=1.25in, right=1.25in]{geometry}
\usepackage[colorlinks=true,pagebackref=true]{hyperref}
\usepackage[textsize=scriptsize]{todonotes}
\hypersetup{backref,hidelinks}

\usepackage{etoolbox}
\newtoggle{final}
\toggletrue{final}

\iftoggle{final} {
\newcommand{\aravind}[1]{} 
\newcommand{\tom}[1]{} 
\newcommand{\mike}[1]{} 
\newcommand{\NB}[1]{}
\newcommand{\TODO}[1]{}
\renewcommand{\todo}[1]{}
}{ 
\newcommand{\aravind}[1]{\textcolor{red}{#1}} 
\newcommand{\tom}[1]{\todo[color=blue!40]{#1}} 
\newcommand{\mike}[1]{\textcolor{green}{#1}} 
\newcommand{\NB}[1]{\todo[color=gray!40]{#1}}
\newcommand{\TODO}[1]{\todo[color=red]{#1}}
}

\newcommand{\colim}{\operatorname{colim}}

\newcommand{\Spec}{\operatorname{Spec}}

\newcommand{\isomto}{{\stackrel{\sim}{\;\longrightarrow\;}}}

\newcommand{\sma}{{\scriptstyle{\wedge}\,}}

\renewcommand{\hom}{\operatorname{Hom}}
\newcommand{\Map}{\operatorname{Map}}
\newcommand{\iMap}{\underline{\Map}}

\newcommand{\Id}{\mathrm{id}}

\newcommand{\real}{{\mathbb R}}

\newcommand{\cplx}{{\mathbb C}}
\newcommand{\C}{{\mathrm C}}
\newcommand{\Q}{{\mathbb Q}}
\newcommand{\Z}{{\mathbb Z}}

\renewcommand{\P}{{\mathrm P}}

\newcommand{\aone}{{\mathbb A}^1}
\newcommand{\pone}{{\mathbb P}^1}

\renewcommand{\1}{{\rm 1\hspace*{-0.4ex}%
\rule{0.1ex}{1.52ex}\hspace*{0.2ex}}}

\newcommand{\gm}[1]{{{\mathbb G}_{m}^{#1}}}
\newcommand{\Gm}{{\gm{}}}
\renewcommand{\L}{{\mathrm L}}

\newcommand{\SH}{{\mathrm{SH}}}

\newcommand{\Sm}{\mathrm{Sm}}

\newcommand{\Spc}{\mathrm{Spc}}

\newcommand{\Pic}{\operatorname{Pic}}

\newcommand{\Sym}{{\operatorname{Sym}}}

\newcommand{\Sing}{\operatorname{Sing}}

\renewcommand{\setminus}{\smallsetminus}

\newcommand{\fib}{\mathrm{fib}}

\newcommand{\Addresses}{{
\bigskip
\footnotesize

A.~Asok, Department of Mathematics, University of Southern California, 3620 S.~Vermont Ave., Los Angeles, CA 90089-2532, United States; E-mail address: \url{asok@usc.edu}
\medskip

T.~Bachmann, Department of Mathematics, JGU Mainz, Staudingerweg 9, 55128 Mainz, Germany; E-mail address: \url{tom.bachmann@zoho.com}
\medskip

M.J.~Hopkins, Department of Mathematics, Harvard University, One Oxford Street, Cambridge, MA 02138, United States \textit{E-mail address:} \url{mjh@math.harvard.edu}
}}

\newcounter{intro}
\theoremstyle{plain}
\newtheorem{theorem}{Theorem}[subsection]

\newtheorem{cor}[theorem]{Corollary}
\newtheorem{proposition}[theorem]{Proposition}
\newtheorem*{claim*}{Claim} 

\newtheorem{question}[theorem]{Question}

\newtheorem*{thm*}{Theorem}
\newtheorem*{problem*}{Problem}
\newtheorem*{question*}{Question}

\newtheorem{thmintro}{Theorem}

\newtheorem{problemintro}[thmintro]{Problem}
\newtheorem{questionintro}[thmintro]{Question}

\theoremstyle{definition}
\newtheorem{defn}[theorem]{Definition}

\theoremstyle{remark}
\newtheorem{rem}[theorem]{Remark}

\newtheorem{ex}[theorem]{Example}

\numberwithin{equation}{subsection}

\begin{document}
\pagestyle{fancy}
\renewcommand{\sectionmark}[1]{\markright{\thesection\ #1}}
\fancyhead{}
\fancyhead[LO,R]{\bfseries\footnotesize\thepage}
\fancyhead[LE]{\bfseries\footnotesize\rightmark}
\fancyhead[RO]{\bfseries\footnotesize\rightmark}
\chead[]{}
\cfoot[]{}
\setlength{\headheight}{1cm}

\author{Aravind Asok \and Tom Bachmann \and Michael J. Hopkins}

\title{{\bf Algebraic vs. analytic: a survey}}
\date{}
\maketitle

\begin{abstract}
	Given a pair of complex algebraic varieties, we can ask: does every analytic map admit an algebraic representative?  Such questions are most interesting when the source and target are non-compact varieties, in which case the problems are closely linked with Hodge-type conjectures.  Concrete avatars of this kind of problem, which we discuss from several points of view, include: which holomorphic vector bundles on an affine variety admit an algebraic structure, and which homotopy classes of maps admit algebraic representatives?  
\end{abstract}

\begin{footnotesize}
\tableofcontents
\end{footnotesize}

\section*{Proemium}
The purpose of this note is deliberately experimental. Mathematical papers ordinarily explain why a theorem is true, and occasionally why a question is natural. Much more rarely do they ask how the question itself became mathematically meaningful. However, the meaning of a mathematical problem is no less historically contingent than that of its solution.  Before a question can be answered, there must exist a mathematical language in which it can even be formulated. That language itself has a history.

Our starting point is the following pair of comparison problems.  To state them, write $S^{n}$ for the ``standard'' $n$-sphere, i.e., the one modelled by the set of real solutions to the equation $\sum_{i=0}^n x_i^2 = 1$; write $\mathrm{S}^{n}_{\cplx}$ for the corresponding nonsingular complex affine variety defined by the same equation.  An elementary computation,  effectively the polar decomposition of complex numbers, shows that the complex manifold attached to $\mathrm{S}^{n}_{\cplx}$ is homotopy equivalent to $S^n$.  Fixing base-points, we may consider the following problem.

\begin{questionintro}
	\label{questionintro:polynomialrep}
	Fix integers $n, m \geq 1$.  Does any element of $\pi_n(S^m)$ admit a representative by a polynomial map of (complex) affine quadrics?  
\end{questionintro}

We also consider the following problem, which has a similar flavor (cf. \cite[Question 3.20]{AsokFasel2022VectorBundles}).

\begin{problemintro}
	\label{problemintro:algebraizability}
Suppose $X$ is a nonsingular complex affine algebraic variety and write $X^{an}$ for $X(\cplx)$ viewed as a complex manifold.  If we write $\mathscr{V}_r(X)$ for the set of rank $r$ algebraic vector bundles on $X$ and $\mathscr{V}_r^{top}(X)$ for the set of isomorphism classes of complex (topological) vector bundles on $X^{an}$, then characterize the image of the natural map:
\[
\mathscr{V}_r(X) \longrightarrow \mathscr{V}_r^{top}(X).
\]    
\end{problemintro}

The principal mathematical contribution of this paper is the following pair of theorems.  Briefly, the first result states that the answer to Question~\ref{questionintro:polynomialrep} is positive.  The second result gives a condition ensuring bijectivity of the comparison map in Problem~\ref{problemintro:algebraizability}.

\begin{thmintro}[A., Bachmann, Hopkins]
	\label{thmintro:polynomialrep}
	For any integers $m,n \geq 1$, every element of $\pi_n(S^m)$ admits a complex polynomial representative.  More precisely, for any integers $n \geq m$ the map
	\[
	\hom_*(\mathrm{S}^{n}_{\cplx},\mathrm{S}^{m}_{\cplx})/\sim_{\aone} \longrightarrow \pi_{n}(S^{m})
	\]
	is surjective; here, the left hand set is the quotient of the set of pointed morphisms of algebraic varieties modulo the equivalence relation generated by algebraic homotopies parameterized by the affine line.\footnote{In fact, the comparison map is usually bijective, and the situations where it fails to be bijective can be described explicitly; see Remark~\ref{rem:polynomialreps}.}  
\end{thmintro}


\begin{thmintro}[A., Bachmann, Hopkins]
	\label{thmintro:cellularalgebraizability}
If $X$ is a smooth complex affine {\em even cellular} (see \textup{Definition~\ref{defn:evencellular}}) variety, then the map
\[
\mathscr{V}_r(X) \longrightarrow \mathscr{V}_r^{top}(X)
\]	
is bijective.  In particular, if $\widetilde{{\mathbb P}^n}$ is the affine variety obtained by taking the complement of the incidence divisor in ${\mathbb P}^n \times {\mathbb P}^n$ (a.k.a. the standard Jouanolou device of ${\mathbb P}^n$), then every topological vector bundle on $\widetilde{{\mathbb P}^n}$ admits a unique algebraic structure.
\end{thmintro}	

We will situate the questions and results discussed here in three distinct senses:
\begin{enumerate}[noitemsep,topsep=1pt] 	
	\item {\em mathematically}, by explaining the ideas underlying the proofs; 	
	\item {\em rationally},\footnote{The choice of terminology is inspired by what are typically called {\em rational reconstructions}; see, e.g., \cite[pp. 258-9]{Mehrtens}: ``The mathematician's understanding of the history of his subject is frequently sharply anti-sociological. But as frequently it is a 'rational reconstruction' of a development governed by universal laws - those of mathematics itself.''} by reconstructing earlier mathematics in modern language and thereby explaining why the questions appear natural from a contemporary perspective;
	\item {\em historically}, by asking how the questions themselves became mathematically meaningful and identifying some of the mathematical, social, and cultural conditions that made their formulation possible.
	\end{enumerate}
These three perspectives are complementary rather than competing. The first is the familiar mode of exposition in contemporary mathematics. The second is perhaps the most common way mathematicians explain the history of their subject to one another. Its great strength is that it constructs a mathematical continuity across disparate periods, allowing earlier work to be understood as progress toward contemporary mathematics. Its principal weakness is precisely the same: by translating the mathematics of the past into contemporary language, it transforms the very questions under consideration. What emerge are not historical problems expressed in modern language, but genuinely modern problems having historical antecedents. The third, while familiar within the history of mathematics, is comparatively rare in contemporary research mathematics. Our aim is not to privilege one of these perspectives over the others, but to place them in dialogue.

\subsubsection*{Acknowledgments/Disclaimer}
These notes are loosely based on lectures given by the first and third authors at the 2025 Summer Algebraic Geometry Research Institute. The main new results described here have had a long gestation: they form part of a program initiated in collaboration with Jean Fasel, and we thank him for his collaboration and many discussions over the years. We also thank Burt Totaro and an anonymous referee for numerous incisive comments that helped shape the revision.

Neither the mathematical narrative nor the historical one should be regarded as definitive. Both reflect our present understanding of questions whose development continues to unfold. We have nevertheless chosen to present this account in its current form, not as the final word on these topics, but in the hope that it contributes to an ongoing conversation about both the mathematics and its history.


\section{Approximating motivic spaces and topological comparison results}
\label{s:mathematics}
Theorems~\ref{thmintro:polynomialrep} and \ref{thmintro:cellularalgebraizability} are comparison statements: the formulation of each involves a map from a source having an algebraic origin to a target housed in topology.  Our starting point is the observation that the target is an object of homotopy theory; this follows by definition in Theorem~\ref{thmintro:polynomialrep}.  In Theorem~\ref{thmintro:cellularalgebraizability} the interpretation of $\mathscr{V}_r^{top}(X)$ as the set of homotopy classes $[X,BU(r)]$, where $BU(r)$ is the classifying space for complex vector bundles, is known as the Pontryagin--Steenrod representability theorem for vector bundles (we will revisit this statement in Section~\ref{ss:fiberbundles}).  

Rather than try to describe this set of homotopy classes directly, we utilize systematic homotopical approximations to the target (i.e., $S^m$ or $BU(r)$ in the cases of interest).  More precisely, these targets are objects of {\em unstable} homotopy theory, while the approximation schemes exploit {\em stable} homotopy theory, which is frequently much more computable because of its intrinsically cohomological nature (perhaps this distinction led Sammy Eilenberg to quip: ``we can distinguish two cases–the stable case and the interesting case'').  These homotopical approximations yield spectral sequences computing the relevant sets of homotopy classes (see Section~\ref{ss:approximatingspaces}).  Along the way, we try to highlight the utility of having different approximation schemes.  

To analyze the source, we invoke a parallel story. Morel--Voevodsky motivic homotopy theory provides a homotopical interpretation of the algebraic source. Section~\ref{ss:motivichomotopytheory} recalls the necessary constructions and reformulates the comparison problems as questions about motivic homotopy classes.  Thus the original comparison problem becomes one of comparing motivic and ordinary homotopy classes.  But homotopy classes are themselves connected components of suitable mapping spaces. Rather than comparing these spaces directly, our strategy is to compare the corresponding approximation schemes.

A first complication arises because the usual notion of connectivity is poorly adapted to motivic phenomena. The weak cellular classes introduced in \cite{ABHFreudenthal} refine connectivity by incorporating the additional "weights" present in motivic homotopy theory.  Such refined connectivity is essential even to formulate a precise statement relating unstable to stable motivic homotopy theory, the latter being substantially more amenable to computation.  This discussion culminates in Section~\ref{ss:stabloccomp}, where we state the motivic Freudenthal suspension theorem, whose name pays homage to its counterpart in topology, which provides basic control over the stabilization procedure.

However, the approximations we aim to employ require more than stabilization alone: to make the resulting approximations more tractable, we exploit arithmetic fracture techniques in the sense of Quillen and Sullivan, together with their motivic analogues.  These constructions decompose complicated spaces into pieces that are individually more homotopically accessible.  Perhaps the most surprising feature from the perspective of algebraic geometry is that even ``reasonable'' motivic spaces are understood by decomposing them into pieces that no longer possess an evident algebro-geometric interpretation.

The preceding constructions become useful only after they are related to topology. This relationship is encoded by the complex realization functor, whose basic properties we outline.  Section~\ref{ss:somaticlevel} develops a bootstrapping procedure for propagating comparison results from well-understood motivic spaces to much larger classes of examples.  The starting point here is a cohomological comparison, namely between the cohomology of $\Spec \cplx$ in the motivic setting and the ordinary cohomology of a point.  With finite coefficients, the latter is reconstructed from the former, and this observation essentially goes back to Bloch.  An axiomatic reformulation of this comparison statement leads us to the notion of {\em somatic level} (which incorporates a refinement of the comparison keeping track of ``weights'' in motivic cohomology).  The bootstrap begins with even cellular spaces, which we introduce here, and which also intercede in the hypotheses on $X$ in the formulation of Theorem~\ref{thmintro:cellularalgebraizability}.  We also introduce a key black-box: the Bachmann--Hopkins resolution of M. Hill and M. Hopkins' Wilson space hypothesis.  Its role here is to provide control of realization for the constituent spaces appearing in our approximation schemes.

Finally, Section~\ref{ss:homotopicaldescent} assembles the preceding ingredients into a comparison argument.  Rather than giving complete proofs of the main theorems, we illustrate the method in a technically simpler setting.  One striking consequence of our approach is that while it is quite difficult to answer questions about whether {\em specific} maps are algebraic, it is comparatively easier to establish this for the {\em totality} of such maps.

\subsection{Approximating mapping spaces}
\label{ss:approximatingspaces}
One basic tool of homotopy theory that has routinely been leveraged to shed light on the geometry of a given space is systematic approximation by spaces with controlled homotopy types.  One familiar example of the kind of construction we have in mind is the Postnikov tower.  Here, one observes that a suitable class of spaces (e.g., simply connected spaces for concreteness) can be well--approximated by Eilenberg--Mac Lane spaces $K(A,n)$ for various $n$; here control stems from the fact that Eilenberg--Mac Lane spaces have at most $1$ non-vanishing homotopy group.  In more detail, given a simply connected space $X$, there is a sequence of spaces $\tau_{\leq n} X$ together with maps $X \to \tau_{\leq n} X$ and maps $\tau_{\leq n}X \to \tau_{\leq n-1} X$, each of which is classified by a map $\tau_{\leq n-1} X \to K(\pi_n(X),n+1)$, and such that the map $X \to \lim_n \tau_{\leq n} X$ is an equivalence.  Such an ``approximation'' essentially always comes with a tradeoff; in this case: information about homotopy groups is typically hard to come by, whereas the approximation is quite good.  Relevant for our discussion later is the fact that the objects $K(A,n)$ are of cohomological nature and therefore objects of stable homotopy theory.  

In specific cases, reinterpretation of the above construction leads to different approximation schemes.  For example, if $X = S^n$, then the Brouwer degree gives an identification $\pi_n(S^n) = \Z$, and therefore the first stage of the Postnikov tower corresponds to a map $S^n \to K(\Z,n)$.  One interpretation of the Dold-Thom theorem is that ``the free abelian group'' on $S^n$, which can be modelled by the infinite symmetric product $\Sym^{\infty} S^n$, is $K(\Z,n)$ (see \cite{DoldThom1958}).  If we model topological spaces by simplicial sets, then the (simplicial) free abelian group on a simplicial model of $S^n$ is itself a model of $K(\Z,n)$, which provides one justification for the terminology.  In this case, the first stage of the Postnikov tower is modelled by the canonical map from a simplicial set to the associated simplicial free abelian group.  The homotopy of the target encodes the total homology of the source, so we can view the target as the best homological approximation to the source.  We can then attempt to measure the failure of this approximation, which is geometrically incarnated in the (homotopy) fiber of the map and repeat the procedure.  In essence, this description can be extracted from the Cartan--Serre method of ``killing homotopy groups''.

Thinking simplicially, one can also see some of the problems inherent in approximation schemes built in this spirit.  For example, if $X$ is a pointed connected space modelled by a simplicial set, then the loop space $\Omega X$ admits a nice model as an explicit group called the Kan loop group, frequently denoted $GX$.  As a (simplicial) group, this has a lower-central series filtration.  The first quotient attached to this filtration is precisely the abelianization of the group, which allows one to investigate homology.   There are always maps from $GX$ to the various quotients by terms in the lower central series, but the limit of this tower of groups only coincides with $GX$ under additional nilpotence hypotheses (see, for example, \cite{Rector1966,BousfieldEtAl1966, BousfieldCurtis1970, BousfieldKan1972} for related discussion and references).  

Here is a standard way to systematize some of the above constructions.  Write $\Spc_*$ for the $\infty$-category of pointed spaces.\footnote{See \cite[\S 1.2.16]{LurieHTT} for relevant discussion; we employ the terminology of \cite{LurieHTT} and \cite{LurieHA} throughout.}  We build the category $\SH$ by formally inverting suspension, i.e., the operation of smash product with $S^1$; objects of this category are typically called spectra.  The resulting category is a symmetric monoidal stable $\infty$-category; the unit for symmetric monoidal structure is the sphere spectrum $\1$.  Given any pointed space $X$, we write $\Sigma^{\infty}$ for the infinite suspension spectrum from $\Spc_* \to \SH$; this functor has a right adjoint $\Omega^{\infty}: \SH \to \Spc_*$; we can think of $\SH$ as a linearization of $\Spc_*$.  

If $E \in \SH$ is a (connective) spectrum, then there is a unit map $X \to \Omega^{\infty}(E \sma \Sigma^{\infty}X_+)$.  The homotopy of the space on the right is identified with the (reduced) $E$-homology of $X$,  also written $E_*(X)$.  One can define $D_1(X) := \fib(X \to \Omega^{\infty}(E \sma \Sigma^{\infty} X_+))$; by construction $D_1(X)$ comes equipped with a map to $X$.  Iterating this construction, i.e., setting $D_0(X) = X$, $D_i(X) := D_1(D_{i-1}(X))$ for $i \geq 1$, we obtain a sequence of spaces and maps as above.  The homotopy exact sequences of this tower of spaces yield an exact couple yielding a spectral sequence (see \cite{BenderskyCurtisMiller1978} for a treatment along these lines; we will revisit this description later).  

For simple $E$, e.g., mod $p$ homology $H\mathbb{Z}/p$ or complex cobordism $MU$, the corresponding spectral sequences are describable: they are called unstable Adams or Adams--Novikov spectral sequences.  Here, the tradeoffs are more subtle.  Computing mod $p$ homology or $MU$-homology of a space is, in principle, easier than attempting to determine homotopy, but the $E_2$-page of the spectral sequence requires understanding the $E$-homology of $X$ as a comodule over the coalgebra of $E$-homology co-operations.  Unlike the case of Postnikov towers above, convergence questions around the resulting spectral sequences are more involved, but well-studied.  After developing some relevant machinery, we aim to expose some similar constructions in motivic homotopy theory, where the definitions are parallel, if slightly more notationally involved.

\subsection{Motivic homotopy theory}
\label{ss:motivichomotopytheory}
Assume $k$ is a field.  We write $\Sm_k$ for the category of smooth $k$-schemes.  The Morel--Voevodsky motivic homotopy theory provides a model homotopy theory for smooth varieties \cite{MV}: one enlarges $\Sm_k$ to be able to ``do homotopy theory'', i.e., so that the resulting category is complete and cocomplete, and then imposes two kinds of relations: forcing $\aone$-invariance and imposing Nisnevich excision.  A key foundational text on the unstable motivic homotopy theory over a field is Fabien Morel's text \cite{MField}, but his 2006 ICM address \cite{MICM} presents something of a panoramic vista. Other surveys of the unstable theory include \cite{AntieauElmanto,WickelgrenWilliamsSurvey}.

Nowadays, the first step is accomplished by enlarging $\Sm_k$ to $\P(\Sm_k)$--the $\infty$-category of presheaves of spaces on $\Sm_k$--via the Yoneda embedding.  The second step is obtained by looking at the full $\infty$-category $\Spc(k)$ spanned by those $\mathscr{X} \in \P(\Sm_k)$ that are (i) $\aone$-invariant, i.e., such that $\mathscr{X}(U) \to \mathscr{X}(U \times \aone)$ is an equivalence for all $U \in \Sm_k$, and (ii) $\mathscr{X}$ turns elementary Nisnevich distinguished squares into (homotopy) pullback squares.  The resulting object, denoted $\Spc(k)$, is called the $\infty$-category of motivic spaces; there is a corresponding localization functor $\L_{mot}: \P(\Sm_k) \to \Spc(k)$ called {\em motivic localization}, which admits an explicit but nevertheless unwieldy in practice description (see, e.g., \cite[\S 3]{Hoyois2017SixOperations} for further details about this description of the $\infty$-category of motivic spaces).  We write $[\mathscr{X},\mathscr{Y}]_{mot}$ for the set of maps in the associated homotopy category of $\Spc(k)$; the latter is the usual Morel--Voevodsky motivic homotopy category.

Details regarding the motivic localization construction can be found in many places, but we refer the reader to \cite[\S 2]{ABHFreudenthal} for a presentation using notation consistent with that used here and summarizing many useful properties of $\L_{mot}$.  This construction is best behaved when $k$ is perfect in the sense that Morel's fundamental results about the structure of homotopy sheaves hold in this generality; see \cite[Introduction]{MField} for statements of the relevant results.  Therefore, it will frequently be convenient to assume $k$ is perfect in what follows, but in the interest of narrative efficiency, we will not indicate where this assumption is strictly necessary.

There are pointed (or based) versions of all the above constructions; we indicate these with additional decorations, e.g., $\P(\Sm_k)_*$ is the $\infty$-category of presheaves of pointed spaces, and $\Spc(k)_*$ is the associated $\infty$-category of pointed motivic spaces.  In these variants, we have analogs of the notion of wedge sum and smash product, which allow us to define analogs of classical spheres.  We write $S^i$ for the topological (or simplicial) $i$-sphere, which can be viewed as an object of $\P(\Sm_k)_*$ by taking the associated constant presheaf.  We write $\gm{\sma j}$ for the $j$-fold smash product of the pointed presheaf $\gm{}$ (with basepoint the identity section $1$).  We then write $S^{i+j,j}_k := S^i \wedge \gm{\sma j}$ (note the indexing).  

There are standard equivalences $\pone_k \sim S^{2,1}$, and ${\mathbb A}_k^{n} \setminus 0 \sim S^{2n-1,n}$ \cite[\S 3.2]{MV}.  Relevant to the statements from the introduction, if we define $\mathrm{Q}_{2n-1}$ to be the subset of ${\mathbb A}^{2n}_k$ with coordinates $(x_1,\ldots,x_n,y_1,\ldots,y_n)$ defined by the equation $\sum_i x_iy_i = 1$, and $\mathrm{Q}_{2n}$ to be the subset of ${\mathbb A}^{2n+1}_k$ with additional coordinate $z$ defined by the equation $\sum_i x_iy_i = z(1-z)$, then it is well-known that $\mathrm{Q}_{2n-1} \sim S^{2n-1,n}$, while $\mathrm{Q}_{2n} \sim S^{2n,n}$ \cite{ADF}.  

\begin{rem}
	Insofar as the discussion of the introduction is concerned, observe that, for each integer $n \geq 1$, there are isomorphisms of $\cplx$-schemes of the form: $\mathrm{S}^{2n-1}_{\cplx} \cong \mathrm{Q}_{2n-1}$ and $\mathrm{S}^{2n}_{\cplx} \cong \mathrm{Q}_{2n}$.
\end{rem}

The Grassmannian $\mathrm{Gr}_r := \colim_N \mathrm{Gr}_{r,r+N}$, where the (filtered) colimit is formed with respect to a standard sequence of inclusion maps, determines a motivic space.  A key property of this motivic space stems from the affine representability theorem, which shows that $\mathrm{Gr}_r$ plays a role in algebraic geometry analogous to, but slightly more restricted than, the role it plays in topology (see Section~\ref{ss:fiberbundles} for related discussion).

\begin{theorem}[Morel, Schlichting, A-Hoyois-Wendt {\cite[Theorem 1]{AHWI}}]
	\label{thm:vbrepresentability}
If $X$ is a smooth affine $k$-scheme, then there is a canonical bijection
\[
[X,\mathrm{Gr}_r]_{mot} \longrightarrow \mathscr{V}_r(X).
\]
\end{theorem}

There are numerous variants of this kind of representability theorem, and we refer the reader to \cite{AHWI,AHWII,AHWIII,AHWOctonion,AQuadrics} for details.  For the purposes of our later exposition, observe that there is a motivic equivalence of the form $BGL_r \to \mathrm{Gr}_r$ \cite[\S 4 Proposition 3.7]{MV}, where $BGL_r$ is the presheaf $U \mapsto BGL_r(\Gamma(U,\mathscr{O}_U))$.  We can also describe isomorphism classes of vector bundles of rank $r$ equipped with a specified trivialization of the determinant; these are represented by the space $BSL_r$, defined analogously.

Another useful class of spaces comes from Voevodsky's motivic Eilenberg Mac Lane spaces.  If $A$ is an abelian group, and $m$ and $n$ are integers, then we write $K(A(m),n)$ for the associated motivic Eilenberg--Mac Lane space; we refer to \cite[\S 3]{VMEM} for a precise definition, but only indicate that the main results of that paper provide a description of this object, motivated by the classical Dold--Thom theorem.  

If $\mathscr{X}$ is a pointed motivic space, then set $\Omega^{i,j}\mathscr{X} := \iMap_*(S^{i,j},\mathscr{X})$ (the pointed internal mapping object).  If $j = 0$, we will frequently write $\Omega^i \mathscr{X}$ for the resulting space.  There are associated notions of fiber and cofiber sequences and we freely use this terminology.  

It will be useful to have refinements of classical notions of connectivity; this can be achieved by employing a variant of Dror's cellularization and nullification techniques \cite{Farjoun}.  While higher connectivity is intuitively described in terms of maps from certain spheres being null-homotopic, the presence of bi-graded spheres in motivic homotopy theory suggests that we should have refined notions of connectivity taking these additional spheres into account.   

To this end, consider the set of morphisms $\{ S^{i,j} \times U \to U | U \in \Sm_k\}$; we call this the set of generating $S^{i,j}$-weak equivalences.  We consider the smallest (strongly) saturated collection of morphisms in $\Spc(k)$ containing the generating weak equivalences and call these $S^{i,j}$-weak equivalences.  We view the resulting $S^{i,j}$-local objects inside $\Spc(k)$ and write $\L^{i,j}$ for the corresponding localization functor; we will refer to this as the $S^{i,j}$-nullification functor ($S^{i,j}$-local spaces will be called $S^{i,j}$-null spaces).

\begin{defn}
	We say that a pointed motivic space $\mathscr{X}$ lies in the {\em weak-cellular class} $O(S^{i,j})$ if $\L^{i,j} \mathscr{X} = \ast$, a map $f$ is an $S^{i,j}$-equivalence if $\L^{i,j}(f)$ is an equivalence and $\mathscr{X}$ is $S^{i,j}$-null if $\mathscr{X} \to \L^{i,j} \mathscr{X}$ is an equivalence.
\end{defn}

The weak cellular class $O(S^{i,j})$ is closed under formation of colimits and cofiber extensions and the collection $S^{i,j}$-null spaces is closed under limits and fiber extensions (assuming the base of the fiber sequence is connected).  A pointed connected motivic space $\mathscr{X}$ is $S^{i,j}$-null if and only if the pointed internal mapping space $\iMap_*(S^{i,j},\mathscr{X})$ is contractible (use the fact that $\mathscr{X}$ is connected, which implies $\ast \to \mathscr{X}$ is an effective epimorphism).  One may conclude that if $\mathscr{X}$ is pointed and connected, then $\mathscr{X}$ is $S^{i,j}$-null if and only if for every integer $l \geq 0$ and every $U \in \Sm_k$, any map $S^{i+l,j} \wedge U_+ \to \mathscr{X}$ is null.  

\begin{rem}
	A terminological warning is in order, especially when keeping algebraic varieties in mind: the weak cellular class $O(S^{i,j})$ contains all suspensions of the form $S^{i,j} \wedge X_+$.  While ``cellular varieties'' can be used to build examples lying in some weak cellular classes, the latter is really better compared with ``connectivity''.  We will explore this more in Example~\ref{ex:cellularclassBSLr}.
	
	Indeed, the weak cellular class $O(S^{n,0})$ consists precisely of $n$-connective spaces, and $S^{n,0}$-null spaces are precisely the $(n-1)$-truncated spaces; in fact the functor $\L^{n,0}$ can be identified with the Postnikov truncation functor, denoted earlier by $\tau_{\leq {n-1}}$.  Intuitively, the weak cellular class $O(S^{p,q})$ consists of spaces that are ``topologically'' $p-q$-connective and ``Tate'' $q$-connective (see \cite[Corollary 3.1.27]{ABHFreudenthal} for a precise statement).
\end{rem}

\begin{ex}
	\label{ex:cellularclassBSLr}
For any integer $n \geq 1$, ${\mathbb P}^n \in O(S^{2,1})$.  This statement is deduced inductively from the existence of cofiber sequences of the form 
\[
{\mathbb P}^{n} \longrightarrow {\mathbb P}^{n+1} \longrightarrow S^{2n+2,n+1}.
\]
Taking a colimit as $n \to \infty$, we see that ${\mathbb P}^{\infty} \in O(S^{2,1})$ as well; thus the motivic Eilenberg MacLane space $K(\Z(1),2)$, which is modelled by ${\mathbb P}^{\infty}$ lies in $O(S^{2,1})$ also.  On the other hand, both nonsingular curves of genus $g > 0$ and $\gm{}$ fail to lie in $O(S^{2,1})$ (or even $O(S^{1,0})$).

If $\nu: V \to X$ is a rank $r$ vector bundle on a smooth scheme $X$, then $Th(\nu) \in O(S^{2r,r})$: this is deduced from the fact that Zariski locally $\nu$ can be trivialized, in which case the relevant Thom space is a suspension of a smooth scheme.  Taking colimits, one sees that $MGL_r$, which is itself a motivic Thom space, lies in $O(S^{2r,r})$.  

There is also a ``quaternionic'' analog of the statement about projective spaces above.  Indeed, Panin and Walter define a sequence of smooth affine varieties $\mathrm{HP}^n := Sp_{2n+2}/(Sp_{2n} \times Sp_2)$ and observe that there are cofiber sequences of the form
\[
\mathrm{HP}^n \longrightarrow \mathrm{HP}^{n+1} \longrightarrow Th(\nu_{n})
\]
where $\nu_n$ is a tautological symplectic line bundle (i.e., rank $2$ vector bundle) over $\mathrm{HP}_n$.  Likewise, we deduce that $\mathrm{HP}^{\infty} \cong BSL_2 \in O(S^{4,2})$.

Finally, we may show that $BSL_r \in O(S^{4,2})$ by understanding the behavior of weak cellular classes under certain restricted classes of fiber sequences (see \cite[Proposition 3.1.23]{ABHFreudenthal} for more details).  In particular, there are fiber sequences of the form
\[
S^{2n-1,n} \longrightarrow BSL_{n-1} \longrightarrow BSL_n 
\]
Since $BSL_2 \in O(S^{4,2})$ and $S^{2n-1,n} \in O(S^{4,2})$ for $n \geq 3$, we can inductively deduce that $BSL_n \in O(S^{4,2})$ as well.  Finally, one can show that $K(\Z(r),2r) \in O(S^{2r,r})$ as well; this is more involved, but see, e.g., \cite[Example 3.3.8]{ABHFreudenthal}).
\end{ex}

\subsection{Stabilization, localization and completion}
\label{ss:stabloccomp}
Three tools that have proven extremely useful in classical homotopy theory are stabilization, localization and completion; all of these tools have motivic analogs.  The stable motivic homotopy category is obtained by formally inverting smashing with the motivic sphere $S^{2,1}$; we write $\SH(k)$ for the resulting category, which is a stable $\infty$-category (see \cite{VICM} for a survey, but \cite[\S 6]{Hoyois2017SixOperations} for a construction in this spirit; the remaining terminology we will introduce briefly here is developed in more detail in \cite[\S 2]{ABHFreudenthal}).  One way of thinking about $\SH(k)$ is that it houses cohomology theories on smooth algebraic varieties that satisfy Nisnevich Mayer--Vietoris, are $\aone$-invariant and have a suspension isomorphism with respect to $S^{2,1}$.  Intuitively, the category $\SH(k)$ has a more ``linear'' nature than $\Spc(k)$.  

There is an induced functor $\Spc(k)_* \to \SH(k)$ sending a pointed motivic space $\mathscr{X}$ to its suspension spectrum $\Sigma^{\infty}_{\pone} \mathscr{X}$.  The sphere spectrum $\1_k$, obtained as the suspension spectrum of the zero sphere $S^0_k$, is the unit of a symmetric monoidal structure on $\SH(k)$.  We write $H\Z$ for the object of $\SH(k)$ corresponding to Voevodsky's motivic cohomology and $\mathrm{MGL}$ for Voevodsky's algebraic cobordism spectrum; these objects can be thought of as commutative rings in $\SH(k)$, and it makes sense to consider modules over these objects.  

We write $\mathrm{DM}_k$ for the category of $H\Z$-modules; this is an object equivalent to Voevodsky's ``big'' derived category of motives, and $\mathrm{Mod}(\mathrm{MGL})$ for the category of $\mathrm{MGL}$-modules.  There is an ``extension of scalars'' functor $\SH(k) \to \mathrm{DM}(k)$, and similarly for $\mathrm{MGL}$-modules.  The former functor is frequently called the motivic Hurewicz functor by analogy with classical topology: it sends the suspension spectrum of a smooth scheme $X$ with disjoint base-point added to its Voevodsky motive $\mathrm{M}(X) = H\Z \wedge \Sigma^{\infty}X_+$.  Analogously, given a motivic space $\mathscr{X}$, we set $\mathrm{MGL}[\mathscr{X}] := \mathrm{MGL} \wedge \Sigma^{\infty}\mathscr{X}$; we could call this the $\mathrm{MGL}$-motive of $\mathscr{X}$.

The functor $\Sigma^{\infty}_{\pone}$ has a right adjoint ``$\pone$-infinite loop space'' functor: 
\[
\Omega^{\infty}_{\pone}: \SH(k) \longrightarrow \Spc(k)_*.
\]  
This functor is rather subtle, and the relationship between a pointed motivic space and the associated suspension spectrum is quite complicated in general.  

In ordinary algebraic topology, stabilization is analyzed via the Freudenthal suspension theorem.  Without additional hypotheses in place, the analog of that theorem in motivic homotopy theory--replacing $S^{1}$ by $S^{2,1}$--is false.  The notion of weak-cellular classes, however, can be used to formulate a suitable motivic analog.  

\begin{theorem}[A-Bachmann-Hopkins {\cite[Theorem 1]{ABHFreudenthal}}]
	\label{thm:freudenthal}
	Assume $k$ is a field that has characteristic $0$.  If $\mathscr{X} \in \Spc(k)$ is a pointed motivic space that lies in $O(S^{p,q})$ with $p -q \ge 2, q \geq 2$, then the fiber of the stabilization map
	\[
	\mathscr{X} \longrightarrow \Omega^{2,1}\Sigma^{2,1} \mathscr{X}
	\]
	lies in $O(S^{a,2q})$ where $a = \min(2p-1,p+2q-1)$.
\end{theorem}

\begin{rem}
	A consequence of the motivic Freudenthal suspension theorem is control of the fiber of the unit of the adjunction $\mathscr{X} \to \Omega^{\infty}_{\pone}\Sigma^{\infty}_{\pone} \mathscr{X}$ (there is an estimate on the weak cellular class of the fiber of this map, similar to the above).
\end{rem} 

Henceforth, we take $k = \cplx$, but much of what we say below will work when $k$ has characteristic $0$ and $-1$ is a sum of squares in $k$.  Localization and completion are well-behaved constructions on $\SH(k)$, and they are given by geometric constructions that induce the corresponding operations on (stable) homotopy sheaves.  In all situations, localization with respect to a set of primes can be obtained as a further Bousfield localization of $\Spc(k)$ or $\SH(k)$; the unstable situation is developed in \cite{AFHLocalization}; we write $\L_{\Q}$ for the corresponding localization functor.  One manifestation of this is the observation due to Voevodsky that rationalizing $\SH(k)$ yields $\mathrm{DM}(k)_{\Q}$.  Completion is more delicate, but the corresponding theory for pointed motivic spaces is developed in \cite{Mattis2024UnstablePCompletion}; we write $\L^{\wedge}$ for the corresponding profinite completion functor.  

Sullivan \cite[p. 30]{Sullivan2004GeometricTopology} lays out the goals of this construction quite explicitly.  Given a finitely generated abelian group $A$, we can form the $\Q$-vector space $A_{\Q}$, the profinite completion $A^{\wedge}$, and both of these constructions map to an ``adelic'' completion $A_{\mathbf{A}}$; the group $A$ then sits in a fiber square: it can be recovered as the fiber product of $A^{\wedge}$ and $A_{\Q}$ over $A_{\mathbf{A}}$.  Localization and completion as above yield space-level refinements of this (abelian) group-theoretic construction.  

In particular, for any motivic space $\mathscr{X}$, there is a commutative square:
\[
\xymatrix{
\mathscr{X} \ar[r]\ar[d] & \L_{\Q}\mathscr{X} \ar[d] \\
\L^{\wedge} \mathscr{X} \ar[r] & \L_{\Q} \L^{\wedge} \mathscr{X}
}
\]
Ideally, the above square would be a fiber square (which we would call an arithmetic fracture square by analogy with Sullivan's construction), but the technical issues alluded to above regarding $\L^{\wedge}$ prevent us from making this assertion in general (see \cite[Theorem 8.7]{Mattis2024ArithmeticFracture} for a precise statement).  Nevertheless, we do have appropriate arithmetic fracture squares for spaces $\mathscr{X} \in O(S^{4,2})$, e.g., $BSL_r$.  As a consequence, to understand such spaces, we need to analyze their rationalizations.  For $BSL_r$, observe that ``universal'' Chern classes correspond to maps in the motivic homotopy category $c_i: BSL_r \to K(\Z(i),2i)$, $2 \leq i \leq r$.  The rationalization of $BSL_r$ is then described in the following result, which can be deduced from \cite[Theorem 5.2.1]{AFHLocalization} by delooping.

\begin{theorem}[A.-Fasel-Hopkins]
	\label{thm:rationalsplittingBSLr}
	If $k$ is a field in which $-1$ is a sum of squares, then the map
	\[
	(c_2,\ldots,c_r): BSL_r \longrightarrow K(\Z(2,4)) \times \cdots \times K(\Z(r),2r)
	\]
	is an equivalence after applying $\L_{\Q}$.
\end{theorem}

\subsection{Complex realization and somatic level}
\label{ss:somaticlevel}
To approach the theorem statement in the introduction, it is necessary to get a more controlled understanding of the completion construction; we do this by explicit comparison.  Ideally, we would like to describe completion for motivic spaces over $\cplx$ whose mod $p$ motivic cohomology is ``controlled by topology''.  We describe this as a two-step procedure: 1) isolate a class of spaces whose motivic cohomology is effectively ``determined by topology'', and 2) describe a ``homotopical descent'' procedure that allows us access to the resulting completion.  We focus on the first point in this section.

To that end, recall that there is a {\em complex realization} functor 
\[
r_{\cplx}: \Spc(\cplx) \longrightarrow \Spc
\]
which is obtained by left Kan extension of the functor sending a smooth $\cplx$-scheme $X$ to the complex manifold $X^{an}$ (see \cite[\S 3.3]{MV} for a discussion of this functor at the level of homotopy categories).  This functor is alternatively called {\em topological realization} or {\em Betti} realization by different authors.  As a left Kan extension, this functor preserves (homotopy) colimits as well as finite products of smooth schemes.

\begin{rem}
If $\mathscr{B}$ is a connected, pointed motivic space, then complex realization enjoys a number of nice properties: it commutes with formation of loop spaces, i.e., $r_{\cplx}(\Omega \mathscr{B}) \cong \Omega r_{\cplx} \mathscr{B}$; it preserves fiber sequences, i.e., if $\mathscr{F} \to \mathscr{E} \to \mathscr{B}$ is a fiber sequence, then $r_{\cplx}(\mathscr{F}) \to r_{\cplx}(\mathscr{E}) \to r_{\cplx}(\mathscr{B})$ is a fiber sequence as well and $r_{\cplx}$ preserves $n$-connective spaces for $n \geq 0$.  For these statements, the assumption that $\mathscr{B}$ is connected allows us to rewrite these fiber sequences in terms of a ``bar construction'' for the action of $G = \Omega \mathscr{B}$ on $\mathscr{F}$, in which case the assertion follows from the fact that $r_{\cplx}$ preserves (homotopy) colimits and finite products (more precisely, combine this observation with the conclusion of \cite[Lemma 2.1.16]{ABHFreudenthal}). 
\end{rem}	

The following comparison result provides the basic ingredient which we aim to bootstrap. 

\begin{theorem}[Bloch, Suslin, Voevodsky]
	\label{thm:BlochKatoconj}
	If $p$ is a prime number, then 
	\[
	H^{*,*}(\Spec \cplx,\Z/p) \cong {\mathbb F}_p[\tau] 
	\]
	where $\tau$ is a class in bi-degree $(0,1)$ corresponding to a choice of primitive $p$-th root of unity.
\end{theorem}

\begin{rem}
This explicit formulation is surprisingly difficult to locate in the literature, and its attribution requires some explanation. Rather than give a proof, we briefly indicate how it emerges from several developments.  Totaro has suggested to us that the statement was already known to Bloch by the time of his paper \cite{Bloch1986}, which introduced the first workable definition of integral motivic cohomology via higher Chow groups. Although \cite[\S 11]{Bloch1986} sketches the relationship between higher Chow groups with finite coefficients and étale cohomology (see also the acknowledgements, where the case of separably closed fields is discussed explicitly), no general statement of the theorem appears there. Some arguments in Bloch's paper were later found to be incomplete, prompting Bloch's correction \cite{Bloch1994}.

A proof can already be extracted from the work of Suslin and Voevodsky. Indeed, \cite[Theorem 8.3]{SuslinVoevodsky1996} identifies singular cohomology of complex algebraic varieties with finite coefficients in cycle-theoretic terms, and combining this result with the Artin--Grothendieck comparison theorem yields the statement above.

From the modern perspective, however, the result is most naturally viewed through motivic cohomology. Voevodsky proved that Bloch's higher Chow groups agree with motivic cohomology, first assuming resolution of singularities (see, e.g., \cite[Theorem 19.1]{MVW}) and shortly thereafter unconditionally \cite[Corollary 2]{Voevodsky2002HigherChow}. Finally, Suslin established \cite[Corollary 4.3]{Suslin2011HigherChow} that higher Chow groups with finite coefficients agree with étale cohomology. Together these results immediately yield the theorem.
\end{rem}

When combined with the suspension isomorphism in motivic cohomology, this result implies that we can easily understand the $\mod p$ motivic cohomology of any motivic sphere.  In essence, we axiomatize the conclusion of these computations, isolating a class of spaces whose behavior under complex realization is tightly controlled; we formulate the definition first in cohomological terms.

\begin{defn}
	\label{defn:somaticlevel}
	A motivic space $\mathscr{X} \in \Spc(\C)_*$ has somatic level $\geq r$ if and only if for all $m > 0, b \ge 0$ the map \[ \bar H^{2b+r'}(\mathscr{X}, \Z/m(b)) \longrightarrow \bar H^{2b+r'}(r_\C(\mathscr{X}), \Z/m) \] is an isomorphism for $r' \le r$ and a monomorphism for $r'=r+1$.	
\end{defn}	

The above definitions can be reformulated via motivic spaces when combined with Voevodsky's motivic Dold-Thom theorem.  More precisely, recall that Voevodsky showed motivic Eilenberg--Mac Lane spaces with finite coefficients are sent, via complex realization, to classical Eilenberg--Mac Lane spaces with finite coefficients; this observation also gives another interpretation of the cycle class map from motivic cohomology to singular cohomology, which we will revisit in Section~\ref{ss:hodgetheory}.  

\begin{theorem}[Voevodsky {\cite[Corollary 3.48]{VMEM}}]
	\label{thm:VoevodskyDoldKan}
	If $A$ is an abelian group and $q \geq i$ are integers, then 
	\[
	r_{\cplx}(K(A(q),i)) \cong K(A,i).
	\]
\end{theorem}

While Definition~\ref{defn:somaticlevel} is phrased cohomologically, the formal properties of spaces of fixed somatic level are more naturally expressed in terms of mapping spaces. This reformulation relies on the adjunction underlying complex realization. To this end, recall that the complex realization functor $r_{\cplx}$ admits a right adjoint. One way to describe this adjoint is as the functor sending a topological space to the corresponding constant presheaf on its singular complex. By a slight abuse of notation, we denote this right adjoint by $\Sing(-)$.  Setting 
\[
\L_{\cplx} := \Sing \circ r_{\cplx},
\] 
the adjunction provides a natural transformation $\Id \to \L_{\cplx}$.  In view of Theorem~\ref{thm:VoevodskyDoldKan}, this provides an equivalent formulation of somatic level entirely in terms of mapping properties.

More generally, if $\mathscr{K}$ is any motivic space, then asking for the comparison map
\[
\Map(\mathscr{X},\mathscr{K}) \longrightarrow
\Map(r_{\cplx}\mathscr{X},r_{\cplx}\mathscr{K})
\]
to be an equivalence is equivalent, by adjunction, to asking that the map
\[
\Map(\mathscr{X},\mathscr{K}) \longrightarrow \Map(\mathscr{X},\L_{\cplx}\mathscr{K})
\]
be an equivalence. Thus comparison with topology can be reformulated entirely within the motivic homotopy category. Moreover, this point of view is formally much more flexible: one may replace a single object $\mathscr{K}$ by any collection of objects, or relativize the construction with respect to a morphism $\mathscr{A}\to\mathscr{X}$.

Applying this observation to the collection of maps
\[
BK(\Z/m(b),2b+r)\longrightarrow
\L_{\cplx}BK(\Z/m(b),2b+r),
\]
and using Theorem~\ref{thm:VoevodskyDoldKan}, we conclude that a motivic space $\mathscr{X}$ has somatic level $\ge r$ if and only if
\[
\Map\!\left(
\mathscr{X},
\fib\!\left(
BK(\Z/m(b),2b+r) \to \L_{\cplx}BK(\Z/m(b),2b+r)\right) \right)
\]
is contractible for every $m>0$ and $b\ge0$. This reformulation is considerably better suited to studying the formal properties of spaces of fixed somatic level.

\begin{rem}[Terminological clarification]
The adjective {\em somatic} derives from the Greek {\em soma}, meaning "body." We chose it to suggest two complementary contrasts.

First, it is intended as a counterpart to motivic. Whatever Grothendieck's original intentions—Manin famously associated the word motif with C\'ezanne \cite[p. 440]{Manin1968}, while Milne argues that no such connection was intended \cite{MilneMotivesOnline}—the adjective motivic has come to evoke a highly structured and idealized world of abstractions. By contrast, somatic is meant to suggest something concrete, corporeal, and ultimately visible through complex realization.

Second, and more importantly for this paper, the terminology reflects a process of reification. As will become apparent in Section~\ref{ss:continuity}, arbitrary continuous functions came to occupy an increasingly abstract position within mathematics during the nineteenth century, while algebraic functions remained distinguished by their finite descriptions. A high somatic level expresses the extent to which this abstract topological information can nevertheless be realized by finite algebraic data. Theorems~\ref{thmintro:polynomialrep} and \ref{thmintro:cellularalgebraizability} illustrate this philosophy, but the notion itself is intended to apply much more broadly.
\end{rem}

We would now like to grow the class of spaces of fixed somatic level.  To this end, observe that essentially by definition the notion behaves well in cofiber sequences.  Bearing that in mind, we make the following definition.

\begin{defn}
	\label{defn:evencellular}
	We will say that a motivic space $\mathscr{X}$ is {\em even cellular} if $\mathrm{M}(\mathscr{X})$ is isomorphic to a sum of objects of the form $\Sigma^{2n,n}H\Z$.  If $\mathscr{M}$ is an $MGL$-module, we will say that {\em $\mathscr{M}$ is $MGL$-pure Tate} if it can be written as a filtered colimit of finite sums of the form $\Sigma^{2n,n} \mathrm{MGL}$.  Finally, we say that a motivic space $\mathscr{X}$ is $\mathrm{MGL}$-pure Tate if $\mathrm{MGL}[\mathscr{X}]$ is so.
\end{defn}

Effectively by construction (via the suspension isomorphism in motivic cohomology) one has the following result about the somatic level of even cellular spaces.

\begin{proposition}
	If $\mathscr{X} \in \Spc(\cplx)_*$ is even cellular, then $\mathscr{X}$ has somatic level $\geq 1$.
\end{proposition}

Of course, there are numerous ``easy'' examples of even cellular motivic spaces, e.g., those obtained from projective space or Grassmannians (more generally, those nonsingular complex projective varieties admitting a paving by affine spaces).  However, and relevant to our future discussions, there is an extremely non-trivial example provided by Bachmann and Hopkins' resolution of the Wilson space hypothesis (which asserts that $\Omega^{\infty}_{\pone}\Sigma^{2i,i}\mathrm{MGL}$ is even cellular); we phrase the result in the following way. 

\begin{theorem}[Bachmann, Hopkins]
	\label{thm:wilsonspacehypoth}
	If $\mathscr{M}$ is an $\mathrm{MGL}$-pure Tate object, then $\mathrm{MGL}[\Omega^{\infty}_{\pone}\mathscr{M}]$ is again $\mathrm{MGL}$-pure Tate.  Moreover, $r_{\cplx}\Omega_{\pone}^{\infty} \mathscr{M} = \Omega^{\infty}r_{\cplx}\mathscr{M}$.
\end{theorem}

In the next section, we show how to combine this observation with an approximation scheme to deduce our main result.

\subsection{Homotopical descent}
\label{ss:homotopicaldescent}
Returning to the sketch at the beginning of this section, we now aim to develop various homotopical approximation schemes in motivic homotopy theory.  We begin by phrasing the setup in terms perhaps more familiar to algebraic geometers.  

If $R \to S$ is a ring map, then descent theory asks for a characterization of the essential image of the extension of scalars from $R$-modules to $S$-modules.  Extension of scalars and the forgetful functor form an adjoint pair, and any adjoint pair gives rise to a comonad (see \cite{Grothendieck1960DescenteI} for the algebro-geometric formulation, and \cite{Beck1967Triples} for a categorical formulation).   The situation is similar, but more complicated, when we consider the adjunction:
\[
\xymatrix{
L: \Spc(k)_* \ar@<.3ex>[r] & \ar@<.3ex>[l] \mathrm{Mod}(\mathrm{MGL})
}
\]
where $L\mathscr{X} := \mathrm{MGL} \wedge \Sigma^{\infty} \mathscr{X}$ and the adjoint is given by $\Omega^{\infty}_{\pone}$.  In the case above, $\mathrm{S} := L\Omega^{\infty}_{\pone}$, viewed as an endofunctor of $\mathrm{Mod}(\mathrm{MGL})$, can be equipped with the structure of a comonad.  For any motivic space $\mathscr{X}$, the object $L\mathscr{X} \in \mathrm{Mod}(\mathrm{MGL})$ carries the structure of a coalgebra for this comonad.  

In general, $\mathscr{X}$ cannot be recovered from $L\mathscr{X}$ equipped with this additional structure, mirroring the situation in classical ring theory.   Indeed, the functor $L$ might invert maps of motivic spaces; we will refer to such maps as $L$-equivalences.  The collection of such maps is well-behaved enough that we may form the localization $\hat{L}\Spc(k)_*$ at the $L$-equivalences, and evidently the functor $\Spc(k)_* \to \mathrm{CoAlg}_{\mathrm{S}}(\mathrm{Mod}(\mathrm{MGL}))$ factors through $\hat{L}\Spc(k)_*$.

Using the coalgebra structure of $L\mathscr{X}$, we can build a cosimplicial resolution of $\mathscr{X}$ via the usual cobar construction; in the classical ring-theoretic setting above, the resulting resolution is what is usually called the Amitsur complex.  In this context, we may form the cosimplicial motivic space:
\[
\mathrm{cobar}(\mathscr{X})^n := (\Omega^{\infty}_{\pone}L)^{n+1} \mathscr{X}.
\]
The object $\mathrm{cobar}(\mathscr{X})$ is by definition an unstable Adams--Novikov resolution of $\mathscr{X}$.  Essentially by construction, the totalization of this space is local for $L$-equivalences, and we may ask whether the associated canonical map $\hat{L}\mathscr{X} \to \operatorname{Tot} \mathrm{cobar}(\mathscr{X})$ is an equivalence.  More strongly, we would like a ``workable'' description of the relevant localization.  By analogy with the situation in ordinary algebraic topology, such an ideal description would be in terms of something related to profinite completion of $\mathscr{X}$.  

If $\mathscr{X} \in O(S^{4,2})$, then a description of ``what the cobar resolution converges to'' can be found in \cite[Theorem 6.51(2)]{BachmannEngelmannMattis2025Monadic} (note that the hypotheses are satisfied by \cite[Definition 6.32, Remark 6.33]{BachmannEngelmannMattis2025Monadic}); in particular, this relies on Theorem~\ref{thm:freudenthal}.  Combining these observations with the discussion of somatic level, and Theorem~\ref{thm:wilsonspacehypoth} one can establish comparison results for $\mathscr{X} \in O(S^{4,2})$ such that $\mathrm{MGL}[\mathscr{X}]$ is $\mathrm{MGL}$-pure Tate.  Rather than describing this procedure in general, we indicate how it fits into the general architecture of the proof of the following result, which provides a special case of Theorem~\ref{thmintro:cellularalgebraizability}.  

\begin{theorem}
	\label{thm:vbtrivdet}
	If $X$ is a connected smooth complex variety that is even cellular, then
	\[
	[X,BSL_r]_{mot} \longrightarrow [X^{an},BSU(r)]
	\]
	is a bijection.
\end{theorem}

At a high level, the proof has three logically distinct ingredients: control of profinite completion via the Wilson space hypothesis~\ref{thm:wilsonspacehypoth}, arithmetic fracture techniques in both the ordinary and motivic settings (Section~\ref{ss:stabloccomp}), and the rational splitting of \(BSL_r\) established in Theorem~\ref{thm:rationalsplittingBSLr}. These are combined by comparing the corresponding fracture squares, while the notion of somatic level developed in Section~\ref{ss:somaticlevel} provides the mechanism for propagating comparison results through the resulting diagrams.

\begin{proof}
	The space $BSL_r \in O(S^{4,2})$ (see Example~\ref{ex:cellularclassBSLr}).  Moreover, the $\mathrm{MGL}$-motive of $BSL_r$ is $\mathrm{MGL}$-pure Tate; the latter is a manifestation of a cellular decomposition of a model of $BSL_r$: it admits a model as the total space of a $\gm{}$-torsor over $\mathrm{Gr}_r$, and is homotopically realized as the fiber of a map $\mathrm{Gr}_r \to {\mathbb P}^{\infty}$; these are geometric reflections of the fact that the $\mathrm{MGL}$-cohomology of $BSL_r$ is a polynomial ring on universal Chern classes $c_2,\ldots,c_r$.  Our aim is to analyze $\Map(X,BSL_r)$ when $X$ is even cellular, and we aim to do this by comparison with complex realization.  
	
	The very definition of somatic level involves motivic cohomology with finite coefficients.  As discussed before the theorem statement, the motivic Adams--Novikov resolution converges to ``profinite completion'' (see the discussion of Section~\ref{ss:stabloccomp} for some discussion of profinite completion of motivic spaces).  In conjunction with the Wilson space hypothesis~\ref{thm:wilsonspacehypoth}, we can control the comparison map between profinite completion of the ``motivic'' mapping space and the associated ``topological'' mapping space.  Precisely, since $X$ has somatic level $\geq 0$ by the even cellularity assumption, the Wilson space hypothesis~\ref{thm:wilsonspacehypoth} can be combined with convergence of the motivic Adams--Novikov spectral sequence to deduce that the natural map:
	\[
	\Map(X,\L^{\wedge} BSL_r) \simeq \Map(r_{\cplx} X, \L^{\wedge}BSU(r)).
	\]
	(where $\L^{\wedge}$ on the RHS is the usual profinite completion for topological spaces) is an equivalence.   
	
	Granted knowledge about profinite completions and complex realization, we leverage the weak cellular class assumption on $BSL_r$ again to deduce the existence of a motivic fracture square of the form:
	\[
	\xymatrix{
	BSL_r \ar[r]\ar[d] & \L_{\Q} BSL_r \ar[d] \\
	\L^{\wedge} BSL_r \ar[r] & \L_{\Q} \L^{\wedge} BSL_r;
	}
	\]
	of course, a corresponding square exists as well replacing $BSL_r$ with $r_{\cplx} BSL_r \equiv BSU(r)$.  This square remains a fiber square after applying $\Map(X,-)$ essentially by definition.  
 	
	In a fashion similar to the analysis of profinite completion, one may establish a corresponding equivalence after $\L^{\wedge}$ is replaced by $\L_{\Q}\L^{\wedge}$.  Putting these observations together with a pasting result for Cartesian cubes, one obtains a Cartesian square of the form
	\[
	\xymatrix{
	\Map(X,BSL_r) \ar[r]\ar[d] & \Map(X,\L_{\Q} BSL_r) \ar[d] \\
	\Map(r_{\cplx}X,BSU(r)) \ar[r] & \Map(r_{\cplx}X,\L_{\Q} BSU(r)).
	}
	\]
	In this formulation, the use of the Wilson space hypothesis has been obscured, but it does clarify what remains to be shown, namely that the left hand map has connected fibers.  The existence of the diagram implies this follows from the assertion that the right hand map has connected fibers; this additional connectivity again leverages the even cellularity of $X$.  
	
	Indeed, by appeal to Theorem~\ref{thm:rationalsplittingBSLr}, we know that $\L_{\Q}BSL_r \equiv K(\Q(2),4) \times \cdots K(\Q(r),2r)$, and a similar statement holds for $BSU(r)$ (e.g., by applying Theorem~\ref{thm:VoevodskyDoldKan}).  Thus, it remains to show that, for all $i \in \Z$ 
	\[
	H^{2i - \epsilon}(X,\Q(i)) \longrightarrow H^{2i-\epsilon}(r_{\cplx}X,\Q)
	\]
	is an isomorphism for all $i$ when $\epsilon = 0$ and an epimorphism when $\epsilon = 1$.  Since $X$ is even cellular, these conditions follow immediately from an analysis of the cycle class map. 
\end{proof}

If $X$ is a smooth complex variety, then write $\mathscr{V}_r^{\circ}(X)$ for the set of isomorphism classes of rank $r$ vector bundles equipped with a fixed trivialization of the determinant; likewise, write $\mathscr{V}_r^{\circ,top}(X)$ for the set of complex topological vector bundles equipped with a (topological) trivialization of the determinant.  There is an analog of Theorem~\ref{thm:vbrepresentability} replacing the Grassmannian by $BSL_r$ (see \cite[Theorem 4.1.1]{AHWII}).  In conjunction with this result, Theorem~\ref{thm:vbtrivdet} can then be interpreted in the following way.  

\begin{cor}
	\label{cor:orientedvbcellular}
	If $X$ is a smooth affine even cellular variety, then the map
	\[
	\mathscr{V}_r^{\circ}(X) \longrightarrow \mathscr{V}_r^{top,\circ}(X)
	\]
	is bijective.
\end{cor}

\begin{rem}
	\label{rem:polynomialreps}
	Theorem~\ref{thmintro:polynomialrep} can be approached in a similar way.  In this case, the rational homotopy type of the ``geometric'' motivic spheres $\mathrm{Q}_{2n-1}$ and $\mathrm{Q}_{2n}$ was described in \cite[\S 5.3]{AFHLocalization}.  If $k$ is a field in which $-1$ is a sum of squares, then the fundamental class $\mathrm{Q}_{2n-1} \to K(\Z(n),2n-1)$ is an equivalence after applying $\L_{\Q}$ \cite[Theorem 5.3.3]{AFHLocalization}.  Likewise, we may consider the squaring map $K(\Z(n),2n) \to K(\Z(2n),4n)$; in that case, the fundamental class in motivic cohomology induces a map $\mathrm{Q}_{2n} \to \fib(K(\Z(n),2n) \to K(\Z(2n),4n))$ that becomes an equivalence after applying $\L_{\Q}$ \cite[Theorem 5.3.11]{AFHLocalization}.  The relevant profinite completions can be approached in the same spirit as the proofs sketched above.  
	
	The comparison map of Theorem~\ref{thmintro:polynomialrep} is not an isomorphism in general because of the presence of copies of $\cplx^\times$ in the relevant motivic homotopy groups.  For example, Morel's computations of motivic homotopy groups of motivic spheres imply that $[S^{2m-2,m-1},S^{2m-1,m-1}]_{\aone} \cong \cplx^{\times}$ \cite{MField}.  A second systematic source of copies of $\cplx^\times$ arises from the Whitehead square of the identity. These are likewise encoded in the rational fiber sequences described above (see also \cite[Theorem~3]{AWW}), and together with the preceding examples they account for all failures of bijectivity.
\end{rem}

The comparison theorems above invite a natural question: how might one have approached these problems before the advent of motivic homotopy theory? The next section deliberately adopts that perspective. Rather than viewing the comparison theorems through the lens of modern homotopy theory, we ask how the mathematical techniques available in the middle of the twentieth century might have led one to formulate and analyze the comparison questions considered here.

\section{Algebraizability problems in the ``classical'' era}
\label{s:rationalreconstruction}
The comparison theorems of Section~\ref{s:mathematics} rely on techniques that became available only much later. It is therefore natural to ask how questions of the kind posed in the introduction might have been approached using the mathematical tools available in the decades after World War II. Our aim in this section is not to reproduce modern proofs in an earlier language, but rather to reconstruct how such questions might naturally have been analyzed using the mathematics of the time.\footnote{We use the adjective "classical" only as a convenient shorthand. Mathematicians routinely apply the term to mathematics lying beyond an implicitly chosen temporal horizon, but such usage inevitably reflects the perspective of the present rather than that of historical actors. We therefore employ the term descriptively rather than as a historical periodization.}

Questions of the kind posed in the introduction were actively investigated between roughly 1950 and 1980, although they appeared in a variety of formulations. Section~\ref{s:genealogy} will return to the historical processes by which such questions became meaningful; for the moment, we deliberately adopt the mathematical viewpoint of the period and ask what its techniques could reveal.

Section~\ref{ss:linebundlecomparison} analyzes algebraizability of line bundles. The exponential sheaf sequence, the divisor--line bundle correspondence, resolution of singularities, and some complementary facts from Hodge theory conspire to give a nearly complete description of the image of the comparison map in this setting.  Already in this simplest case, the comparison map can fail to be either injective or surjective.

In Section~\ref{ss:hodgetheory} we analyze vector bundles of higher rank. Chern classes and the topology of affine varieties yield an ``easy''/``hard'' dichotomy: for varieties of low dimension, obstruction theory reduces the comparison problem essentially to the case of line bundles, whereas in higher dimensions Chern classes no longer contain sufficient information.

Section~\ref{ss:quadraticmaps} concerns a parallel analysis regarding polynomial representatives of homotopy classes. Work surrounding the Bott periodicity theorem suggested that algebraic approximation might provide access to stable homotopy theory, while results of Baum and Wood revealed a sharp distinction between real spheres and their complexifications. The corresponding holomorphic problem, by contrast, falls within the scope of Oka-type principles.

Finally, Section~\ref{ss:motivicresults} returns to motivic homotopy theory. Modern methods both explain the limitations of the preceding approaches and extend them far beyond what mid-century techniques could achieve. Taken together, the discussion isolates those parts of the problems that could already be treated with mid-century techniques and those for which genuinely new structures are required.

\subsection{Algebraic vs. topological line bundles, revisited}
\label{ss:linebundlecomparison}
By the end of the 1950s with sheaf cohomology firmly within the mathematical mainstream, the analysis of Problem~\ref{problemintro:algebraizability} from the introduction, a closely related form of which appears in some work of Serre, could be studied in the case $r = 1$ by separately analyzing algebraic and topological line bundles and observing compatibility of these studies.  We begin with an analysis of topological and analytic line bundles on complex manifolds, bringing together sheaf theory and facts about the topology of algebraic varieties. 

\subsubsection*{Line bundles and sheaf cohomology}
If $X$ is a complex manifold, then locally trivial complex topological line bundles on $X$ are classified as elements of the sheaf cohomology group $H^1(X,\C^{\times}_X)$ where $\C^{\times}_X$ is the sheaf of non-zero complex valued continuous functions on $X$.  If $\Z_X$ is the constant sheaf on $X$ associated with the abelian group $\Z$, then we may write down the exponential sheaf sequence
\[
0 \longrightarrow \Z_X \longrightarrow \C_X \stackrel{exp(2\pi i -)}{\longrightarrow} \C^{\times}_X \longrightarrow 0.
\]   
We consider the following portion of the associated long exact sequence in cohomology:
\[
\cdots \longrightarrow H^1(X,\C_X) \longrightarrow H^1(X,\C^{\times}_X) \stackrel{c_1^{top}}{\longrightarrow} H^2(X,\Z_X) \longrightarrow H^2(X,\C_X) \longrightarrow \cdots,
\]
where we have written $c_1^{top}$ for the connecting homomorphism in the long exact sequence.  The sheaf $\C_X$ is a soft sheaf and therefore its higher cohomology vanishes \cite{Godement1958Topologie}, showing that $c_1^{top}$ is an isomorphism of abelian groups.  The sheaf cohomology group $H^2(X,\Z_X)$ can be identified with singular cohomology $H^2(X,\Z)$ and one deduces that topological complex line bundles are in bijection with elements of $H^2(X,\Z)$.

\begin{ex}
	\label{ex:toplinebundlesoncurves}
	To see how to use the above, suppose $X$ is a nonsingular complex affine curve. Then, $X$ has a nonsingular projective compactification $\bar{X}$, which is a compact Riemann surface of some genus $g$.  In that case, $\bar{X} \setminus X$ is a (non-empty) finite set, and one knows that $X$ has the homotopy type of a wedge of finitely many circles.  In that case, $H^2(X,\Z)$ necessarily vanishes, and $H^1(X,\Z)$ is a free abelian group equal to the number of circles in the wedge, which can be explicitly identified in terms of the genus $g$ and the number of points being removed.  In any case, we conclude that every topological complex line bundle on a nonsingular complex affine curve is trivial. 
\end{ex}

\begin{rem}
	\label{rem:CartanOka}
	Henri Cartan gave the name Stein manifold to a class of complex manifolds whose properties were first analyzed by Karl Stein in 1951 \cite{Stein1951}; Stein manifolds can be defined as closed complex submanifolds of $\cplx^n$ (this is not Stein's original definition, but is equivalent to it by a 1956 result of R. Remmert \cite{Remmert1956CRAS}).  If $X$ is a Stein manifold, then there are holomorphic variants of the statements above: one may replace $\C_X$ with $\mathscr{O}^{hol}_X$ the sheaf of holomorphic functions, and $\C_X^{\times}$ with $\mathscr{O}^{hol,\times}_X$ the sheaf of non-zero holomorphic functions.  In that case, there is a corresponding exponential sheaf sequence exactly as above.  The sheaf $\mathscr{O}^{hol}_X$ is a coherent analytic sheaf and the assumption on $X$ allows one to apply Cartan's theorems A and B \cite{CartanBruxelles} to conclude that $\mathscr{O}^{hol}_X$ has vanishing higher cohomology (in fact, holomorphic formulations like these predate the ``continuous'' discussion above; they appear already in a series of letters from Serre to Cartan from 1952, reproduced in \cite{Serre1991PetitsCousins}; Cartan's theorems are phrased here first as  ``naive questions'', but are used as hypotheses in later letters.).  We deduce that
	\[
	H^1(X,\mathscr{O}^{hol,\times}_X) \longrightarrow H^2(X,\Z_X)
	\]
	is a bijection, which is Cartan's formulation of the Oka principle: every topological complex line bundle admits a unique analytic structure (in particular, in light of Example~\ref{ex:toplinebundlesoncurves}, every analytic line bundle on a nonsingular affine curve is necessarily trivial).  One might view this result in the spirit of Weierstrass's approximation theorem (in the non-compact case). Grauert established a generalization of this result, which we will recall shortly in Theorem~\ref{thm:grauertoka}. 
\end{rem}

\subsubsection*{Algebraic line bundles and divisors}
The lack of an exponential sheaf sequence in the algebraic context necessitates a slightly different approach to the treatment of algebraic line bundles.  Instead, if $X$ is a connected, nonsingular complex algebraic variety, then we have the divisor--line bundle correspondence, which we may formulate as follows (see Section~\ref{ss:algvscont} for additional context).  Write $\cplx(X)$ for the field of rational functions on $X$.  In that case, any codimension $1$ point $x$ of $X$ defines a discrete valuation $ord_x$ of $\cplx(X)$ computing the order of vanishing of $f \in \cplx(X)$ at this point.  In that case, we obtain a $2$-term complex 
\[
\cplx(X) \stackrel{\operatorname{div}}{\longrightarrow} \bigoplus_{x \in X^{(1)}} \Z
\]
where $X^{(1)}$ is the set of codimension $1$ points of $X$, and the map $\operatorname{div}$ sends a rational function $f$ to $\sum_{x \in X^{(1)}} ord_x (f)$, which is a finite sum.  The cokernel of the above map is, by definition, $CH^1(X)$, the divisor class group of $X$.  Under the assumptions on $X$, the above complex can be viewed as a flasque resolution of the sheaf $\Gm$ in the Zariski topology; we thus obtain an identification $CH^1(X) \isomto  H^1(X,\Gm) =: \Pic(X)$.  Given a line bundle $L$ on a nonsingular complex algebraic variety, we write $c_1(L) \in CH^1(X)$ for the associated class. 

If $X$ is a curve, then codimension $1$ points on $X$ are closed points in the usual sense.  In that case, the description above can be effectively used in conjunction with the fact that $X$ admits a unique nonsingular projective model $\bar{X}$.  In this case, the set $D := \bar{X} \setminus X$ is a (possibly empty) finite set, and there is an exact sequence of the form
\[
\Z^{|D|} \longrightarrow CH^1(\bar{X}) \longrightarrow CH^1(X) \longrightarrow 0.
\]
The algebraic line bundles on $\bar{X}$ can be described in several ways, corresponding to different ways of thinking about the Jacobian variety for $\bar{X}$.  Let us write $g$ for the genus of $\bar{X}$.

If one knows that analytic and algebraic line bundles coincide, as was established in Serre's famous GAGA paper \cite{SerreGAGA}, then the sheaf cohomological description of line bundles follows from the analysis of the exponential sheaf sequence.  Indeed, the compactness of $\bar{X}$ tells us that global invertible functions are constant; in particular this means $H^0(\bar{X},\mathscr{O}_{\bar{X}}^{hol}) \cong \cplx$ and $H^0(\bar{X},\mathscr{O}_{\bar{X}}^{hol,\times}) \cong \cplx^{\times}$.  Since the latter is a divisible group and $H^1(\bar{X},\Z)$ is a free abelian group of rank equal to twice the genus of $\bar{X}$, the connecting homomorphism is necessarily the trivial map, and one concludes that there is an exact sequence of the form: 
\[
0 \longrightarrow H^1(\bar{X},\mathscr{O}_{\bar{X}}^{hol})/H^1(\bar{X},\Z) \longrightarrow H^1(\bar{X},\mathscr{O}_{\bar{X}}^{hol,\times}) \longrightarrow \Z \longrightarrow 0.
\]
The Riemann--Roch theorem tells us that $\chi(\mathscr{O}_{\bar{X}}^{hol}) = 1-g$, which unwinding the definition of the Euler characteristic implies $H^1(\bar{X},\mathscr{O}_{\bar{X}}^{hol})$ is a complex vector space of dimension equal to $g$.  We conclude that analytic line bundles of a fixed degree (first Chern class) on $\bar{X}$ are parameterized by a complex torus of dimension $g$.

\begin{rem}
There were also ``purely algebraic'' descriptions of this identification. Weil showed in \cite{Weil1948} that divisors of degree $0$ on $\bar{X}$ were themselves parameterized by an abelian variety of genus $g$; this approach used the Riemann-Roch theorem, but did not rely on complex analysis.
\end{rem}


\subsubsection*{Topological restrictions on algebraic line bundles}
We collect the results described above about the structure of line bundles on nonsingular affine curves in the following result.

\begin{proposition}
	If $X$ is a nonsingular complex affine curve of genus $g \geq 1$, then $\Pic(X)$ is non-trivial and $H^2(X,\Z) = 0$; in particular, the map 
	\[
	\mathscr{V}_1(X) \longrightarrow \mathscr{V}_1^{top}(X)
	\]
	fails to be injective. 
\end{proposition}

To transfer the above argument to analyze higher dimensional nonsingular algebraic varieties, we need some additional results.  We relied on the existence of a nonsingular model of $\bar{X}$; after various geometric arguments that were later viewed as unconvincing, resolution of singularities of surfaces was established by R. Walker in \cite{Walker1935ReductionSingularities} using analytic techniques.   A new ``arithmetic'' argument was soon thereafter given by Zariski \cite{Zariski1939ReductionSingularities} and simplified in \cite{Zariski1942SimplifiedResolution}, and Zariski later established resolution of singularities for threefolds in characteristic $0$. Of course, resolution of singularities for varieties in characteristic zero was not established in all dimensions until the work of Hironaka in the early 1960s.

Bearing that in mind, if $X$ is a nonsingular affine variety of dimension $\leq 2$, then we can always find a nonsingular (projective) compactification $\bar{X}$ such that $\partial X := \bar{X} \setminus X$ is a simple normal crossings divisor.  In that case, Hodge's theory of harmonic representatives gives a decomposition $H^2(\bar{X},\cplx) = \bigoplus_{i+j=2} H^{i,j}(\bar{X})$, where $H^{i,j}(\bar{X}) := H^j(\bar{X},\Omega^i_{\bar{X}})$.  We set $H^{(1,1)}(\bar{X},\Z) := H^2(\bar{X},\Z) \cap H^{1,1}(\bar{X})$.  If $L$ is a complex algebraic line bundle on $\bar{X}$, then the interpretation of Lefschetz's theorem due to Hodge and Weil can be formulated as the assertion that $c_1^{top}(L) \in H^{(1,1)}(\bar{X},\Z)$ and every class in this group is realized as the first Chern class of a complex algebraic line bundle.

Let us write $|\partial X|$ for the (finite) number of irreducible components of $\partial X$.  In that case, the divisor--line bundle correspondence yields the localization exact sequence
\[
\Z^{|\partial X|} \longrightarrow CH^1(\bar{X}) \longrightarrow CH^1(X) \longrightarrow 0
\]
where the right hand map corresponds to restriction of line bundles.  Evidently, the restriction map for line bundles is compatible with the passage from algebraic to topological line bundles.  In conjunction with the discussion of the previous paragraph, there are induced restrictions on the first Chern classes of algebraic line bundles on $X$ as well: we can define $H^{1,1}(X,\Z)$ to be the image of $H^{1,1}(\bar{X},\Z)$ under restriction from $\bar{X}$ to $X$, and first Chern classes of line bundles evidently land here. 

\begin{ex}[Failure of surjectivity]
	\label{ex:surjectivityfails}
	In fact, there exist nonsingular affine surfaces $X$ for which $\mathscr{V}_1(X) \to \mathscr{V}_1^{top}(X)$ fails to be surjective.  For example, let $\bar{X}$ be a K3 surface (note that Andr\'e Weil coined the term K3 surface only around 1958; see \cite[p. 390]{WeilII}).  One knows that $H^2(\bar{X},\mathscr{O}_{\bar{X}}) \cong H^0(\bar{X},\omega_{\bar{X}}) \cong \cplx$.  It is well-known that $H^2(\bar{X},\Z)$ is a free abelian group of rank $22$, and  $H^{1,1}(\bar{X})$ is a complex vector space of dimension $20$.  The group $H^{(1,1)}(\bar{X},\Z)$ is the N\'eron--Severi group, which is an abelian group of rank $\leq 20$.  Now, let $D$ be an ample divisor on $\bar{X}$ so that $X := \bar{X} \setminus D$ is an affine variety.  Bertini's theorem implies that we may assume $D$ is supported on a nonsingular curve in $\bar{X}$ if we wish.  Surjectivity of the map $CH^1(\bar{X}) \to CH^1(X)$ implies any algebraic line bundle on $X$ is, up to isomorphism, the restriction of a line bundle on $\bar{X}$.  In particular, the topological first Chern class of any algebraic line bundle on $X$ lies in the image of $H^{(1,1)}(\bar{X},\Z)$ in $H^2(X,\Z)$. The Gysin exact sequence for the inclusion of $D$ in $\bar{X}$ takes the form:
	\[
	\Z \cong H^0(D,\Z) \longrightarrow H^2(\bar{X},\Z) \longrightarrow H^2(X,\Z) \longrightarrow H^1(D,\Z) 
	\]
	Consequently $\operatorname{rk}(H^2(X,\Z))$ $\geq 21$.  Since the image of $H^{(1,1)}(\bar{X},\Z)$ has rank at most $19$, we conclude the desired non-surjectivity.
\end{ex}


The method analyzed above can be generalized.  If $X$ is a nonsingular complex affine variety, then resolution of singularities guarantees that there is a nonsingular complex projective compactification $\bar{X}$ of $X$ such that $\bar{X} \setminus X$ is a simple normal crossings divisor.  This imposes Hodge-theoretic restrictions on the possible Chern classes of line bundles.  We may define a group $H^{(1,1)}(X,\Z)$ to be the image of $H^{(1,1)}(\bar{X},\Z)$ under restriction; the resulting image does not depend on choice of compactification.  Bearing this in mind, the following statement encapsulates the mechanism by which failure of surjectivity is detected in Example~\ref{ex:surjectivityfails}.

\begin{theorem}
	\label{thm:algebraiclinebundlesonaffinevars}
	If $X$ is a nonsingular affine variety then the image of the map $\mathscr{V}_1(X) \to \mathscr{V}_1^{top}(X)$ consists precisely of those line bundles with first Chern class lying in $H^{(1,1)}(X,\Z)$. 
\end{theorem}

\subsection{Hodge-theoretic restrictions}
\label{ss:hodgetheory}
From a modern point of view, it is natural to ask whether the theorem of the previous section admits an analog for vector bundles of higher rank.  Schwarzenberger studied related questions in detail for surfaces in the early 1960s \cite{Schwarzenberger1961VectorBundlesP2,Schwarzenberger1961VectorBundles}.  A more systematic comparison can be traced to work of Griffiths in the late 1960s and early 1970s \cite{Griffiths1972FunctionTheoryIA,Griffiths1972FunctionTheoryIB}, though there are other sources more in the spirit of Schwarzenberger's work that we will visit momentarily.

\subsubsection*{Chow rings and algebraicity}
To formulate the statement, we need to recall that, in the course of establishing his version of the Riemann--Roch theorem, Grothendieck introduced Chern classes in Chow theory \cite{Grothendieck1958Chern} generalizing the construction for line bundles above.  Grothendieck takes an axiomatic approach to his description, and rather than repeating that construction here, we use a different one, more in line with the homotopical spirit of this text.  Already in \cite[Chapter III.1]{Chern1946}, Chern describes characteristic classes in universal terms via cohomology of Grassmannians, and this description was given a textbook treatment in \cite[\S 14]{MilnorStasheff1974}.  In this account, effectively we simply have to replace cohomology by Chow groups.    

The Chow groups of a nonsingular variety can be defined in a way formally analogous to our ``algebraic'' description of the Picard group via the divisor-line bundle correspondence.  Indeed, if $X$ is a nonsingular variety of dimension $d$ over an algebraically closed field $k$, then we may define $CH^i(X)$ as the cokernel: 
\[
\bigoplus_{x \in X^{(i-1)}} k(x)^{\times} \stackrel{\partial}{\longrightarrow} \bigoplus_{x \in X^{(i)}} \Z \cdot x;
\]
here $k(x)$ is the function field of the variety obtained as the closure of the codimension $i-1$ point $x \in X^{(i-1)}$, and $\partial$ is described in a fashion analogous to the divisor map; this description arises naturally via so-called Gersten resolutions, which appear regularly in the analysis of sheaves arising in motivic homotopy theory, and is one reason we mention it here; see \cite{Rost1996ChowGroupsCoefficients} for a systematic development of these ideas.  Of course, this approach has downsides: verifying many useful properties of Chow groups involves non-trivial complications in this approach (e.g., contravariant functoriality with respect to pullbacks along maps of nonsingular varieties, $\aone$-invariance, etc.)

Since we aim to focus on nonsingular affine varieties, we freely employ Serre's dictionary between (finite rank) algebraic vector bundles on such varieties and finitely generated projective modules over the coordinate ring \cite{SerreFAC}.  One consequence of this dictionary is embedded in the ``functor of points'' description of the Grassmannian variety.  If $X$ is nonsingular and affine, then morphisms $f: X \to Gr_{r,N}$ are in bijection with $N$-generated projective modules of rank $r$ over $R:= \Gamma(X,\mathscr{O}_X)$.  

One may directly compute the Chow ring of the Grassmannian: it is generated by classes $c_i \in CH^i(\mathrm{Gr}_{r,N}), i = 1,\ldots,r$.  In particular, if $E$ is a vector bundle on a nonsingular affine variety $X$, then by choosing generating sections of $E$, we get a map $f: X \to \mathrm{Gr}_{r,N}$, and thus by formal properties of Chow rings, associated classes $c_i(E) := f^*c_i$.  A priori, these classes depend on the particular choice of generating sections, but an $\aone$-homotopy invariance argument (for Chow groups) shows that in fact the Chern classes are independent of the choice.  

\begin{rem}
	In fact, one can define Chern classes for arbitrary nonsingular algebraic varieties in this fashion.  If $X$ is such a variety, then a result of Jouanolou--Thomason shows that there exists a nonsingular affine variety $\tilde{X}$ and an affine morphism $\pi: \tilde{X} \to X$ that is Zariski locally trivial with affine space fibers.  If $\mathscr{E}$ is any vector bundle on $X$, then we may define Chern classes as above for $\pi^*\mathscr{E}$ (see, e.g., \cite[Proposition 4.4]{Weibel1989}, noting that non-singular varieties carry an ample family of line bundles).  Homotopy invariance of Chow groups implies that the Chern classes so defined are independent of choice of $\tilde{X}$.  
\end{rem}

Relevant to our discussion here is the existence of cycle class maps $cl_i: CH^i(X) \to H^{2i}(X,\Z)$.  When $X$ is a complex variety, one standard approach to existence of this map relies on resolution of singularities: given a generating (irreducible) algebraic cycle $Z \subset \bar{X}$, one takes a resolution of singularities $\tilde{Z} \to Z$, and defines $cl_i(Z)$ to be the pushforward of the fundamental class of $\tilde{Z}$.  It can be checked that this class is independent of the resolution, and behaves well with respect to rational equivalence.  It will be more convenient for us to rely on Voevodsky's Theorem~\ref{thm:VoevodskyDoldKan} in the case $A = \Z$: observe that $K(\Z(n),2n)$ represents Voevodsky's motivic cohomology and $H^{2n}(X,\Z(n)) \cong CH^n(X)$ by Voevodsky's comparison between motivic cohomology and Chow groups (see, e.g., \cite[Corollary 19.2]{MVW}).  In that case, the existence of the cycle class map follows from the properties of the complex realization functor we outlined. 

\begin{defn}
	If $X$ is a nonsingular affine variety, then we will say that $u \in H^{2i}(X,\Z)$ is algebraic if it lies in the image of the cycle class map $CH^i(X) \to H^{2i}(X,\Z)$.
\end{defn}

In the particular case of Grassmannians, an immediate consequence of our discussion is that the (universal) Chern classes $c_i$ are mapped to their topological counterparts $c_i^{top} \in H^{2i}(\mathrm{Gr}_{r,N},\Z)$ (indeed, the cycle class map is bijective in this case).  In the context of the above definition of algebraicity, this leads to the following observation.

\begin{proposition}
	If $X$ is a nonsingular affine variety and $E$ is a rank $r$ vector bundle on $X$, then the topological Chern classes $c_i^{top}(E)$ are algebraic.
\end{proposition}

The link between this observation and the Hodge-theoretic restrictions on cohomology classes we described above might seem somewhat tenuous at this stage, but Grauert formulated an extension of the results described in Remark~\ref{rem:CartanOka} that allows us to bring the discussion closer to Hodge theory.

\begin{theorem}[Grauert's Oka principle {\cite{Grauert1958Analytische}}]
	\label{thm:grauertoka}
	If $X$ is a Stein manifold, then for every integer $r > 0$ the  map
	\[
	\mathscr{V}_r^{hol}(X) \longrightarrow \mathscr{V}_r^{top}(X)
	\]
	is bijective.
\end{theorem}

In conjunction with the analysis of Atiyah--Hirzebruch we parenthetically mentioned earlier \cite{AtiyahHirzebruch1962Analytic}, Cornalba and Griffiths used the above results to establish that every rational cohomology class on a Stein manifold (e.g., nonsingular complex affine variety) is represented by the fundamental class of an analytic submanifold \cite[Theorem VIII p.91]{CornalbaGriffiths1975Analytic}.  In view of Theorem~\ref{thm:algebraiclinebundlesonaffinevars}, it thus becomes natural to pose the following question, which is a fundamental question about the structure of algebraic vector bundles on nonsingular affine varieties, roughly in the spirit of the Hodge conjecture. 

\begin{question}
	\label{question:griffiths}
	Does the image of $\mathscr{V}_r(X) \to \mathscr{V}_r^{top}(X)$ consist precisely of those vector bundles whose Chern classes are algebraic? 
\end{question}

There are some standard reductions we can employ to analyze this question.  Indeed, it suffices to assume that $r \leq \dim X$: Serre's splitting theorem \cite{SerreSplitting} implies that every rank $r > d$ algebraic vector bundle on $X$ is isomorphic to the direct sum of a vector bundle of rank $\leq d$ and a trivial bundle of rank $r-d$.  In view of Theorem~\ref{thm:algebraiclinebundlesonaffinevars}, it follows that we only need to analyze the case where $2 \leq r \leq \dim X$.  

\subsubsection*{Topology of affine varieties}
The observation that a nonsingular affine curve has the homotopy type of a wedge of circles suggests that one look at the topological types of nonsingular complex affine varieties.  One famous result in this direction is the Andreotti--Frankel theorem, established in 1958 in the course of giving a new proof of Lefschetz's theorem on the topology of hyperplane sections; we record these results here for our later discussion.  

\begin{theorem}[Andreotti-Frankel {\cite[Theorem 1]{AndreottiFrankel1959}}]
	If $X$ is a Stein manifold (e.g., a nonsingular affine variety) of dimension $d$, then $H_i(X,\Z) = 0$ for $i > d$.
\end{theorem}

In fact, Andreotti--Frankel's homological statement was quickly placed within a broader Morse-theoretic perspective on Lefschetz-type theorems. In the introduction to his 1959 paper \cite{Bott1959Lefschetz}, Bott observes that ``the proper dual of the Lefschetz theorem states that the homology of a Stein manifold vanishes above the middle dimension'', referring to forthcoming work of Andreotti and Frankel using Morse theory, but noting that ``it does not yield the homotopy statement''. Bott remarks that he learned these ideas from Thom and notes the historical irony that Morse and Lefschetz had worked only a few steps from one another for nearly twenty years, yet the application of Morse theory to Lefschetz-type theorems was comparatively recent. The corresponding homotopy-theoretic formulation for Stein manifolds was later popularized by Milnor in his exposition of Morse theory \cite[Theorem~7.2]{Milnor1963}, where he observes that a Stein manifold has the homotopy type of a CW complex of dimension at most its complex dimension.


\subsubsection*{Comparison maps for low-dimensional varieties}
Returning to the algebraicity question for surfaces, the only interesting case is $r = 2$, so let us analyze this case more directly.  If $X$ is a nonsingular complex affine surface, then we may analyze $[X,BU(2)]$ by obstruction theory.  Since $X$ has the homotopy type of a CW complex of dimension $\leq 2$, it follows from a straightforward obstruction theory argument that, topologically, every rank $r > 1$ complex vector bundle on $X$ splits as the sum of a complex line bundle and a trivial bundle of rank $r-1$ (in fact, the obstruction theory argument goes through with only appeal to the Andreotti--Frankel theorem).  In other words, to answer the above question, it suffices to treat the case $r = 1$, which is the content of Theorem~\ref{thm:algebraiclinebundlesonaffinevars}.  

A similar argument can be made for nonsingular affine varieties of dimension $\leq 3$.  Here, one has to treat ranks $2$ and $3$, but once again, the homology of $X$ vanishes in degrees $> 3$.  Thus, an obstruction theory argument as above again reduces the problem to the case of complex line bundles, where yet again we may appeal to Theorem~\ref{thm:algebraiclinebundlesonaffinevars} to conclude.  We summarize these results in the following omnibus theorem.

\begin{theorem}
	If $X$ is a nonsingular affine variety of dimension $\leq 3$, then the image of the map $\mathscr{V}_r(X) \to \mathscr{V}_r^{top}(X)$ consists precisely of those vector bundles with algebraic Chern classes.
\end{theorem}

There is a large circle of results from which the above analysis is culled; we highlight some of these results here.  As we mentioned above, Schwarzenberger studied vector bundles on nonsingular projective surfaces.  His \cite[Theorem 8]{Schwarzenberger1961VectorBundles} sits in a limbo between the algebraic and topological worlds: he shows that given any $c_1 \in CH^1(X)$ and any integer $c_2$, there exists an algebraic vector bundle of rank $r \geq 2$ with given $c_1$ and $c_2$, but he is more broadly interested in the construction of higher rank vector bundles.  Of course, this discussion ignores complications stemming from the complexity of $CH^2(X)$ for surfaces. 

In \cite{MurthySwan1976VectorBundlesAffineSurfaces}, Murthy and Swan analyze affine surfaces using a mix of techniques from pure algebra and homotopy theory.  The ``main result'' stated in the introduction focuses on cancellation questions, i.e., when stable isomorphism of bundles implies isomorphism, which gives a K-theoretic framing to the results.  Along the way, they confront subtleties stemming from the complexity of $CH^2(X)$ in the affine case, but they formulate these complexities in terms of filtrations on $K_0(X)$ (see, e.g.,  \cite[Theorem 4.2]{MurthySwan1976VectorBundlesAffineSurfaces}).

Around the same time, Atiyah and Rees \cite{AtiyahRees1976VectorBundlesP3} extend parts of Schwarzenberger's analysis to ${\mathbb P}^3$ using ideas and techniques of Horrocks and Vogelaar (see the references therein).  After observing that topological vector bundles are not uniquely determined by Chern classes, they introduce a new $\mod 2$ invariant they call the $\alpha$-invariant.  Using index theory, they interpret the $\alpha$-invariant in terms of sheaf cohomology, and then observe that every topological vector bundle on ${\mathbb P}^3$ is algebraizable.  

In the early 1980s, Mohan Kumar and Murthy extended the Murthy--Swan analysis to dimension $3$ \cite{KumarMurthy1982AlgebraicCyclesAffineThreefolds}.  In particular, they establish that if $X$ is nonsingular affine of dimension $3$, then there exists a unique rank $3$ vector bundle with given $(c_1,c_2,c_3) \in \prod_{i=1}^3 CH^i(X)$; they also established existence of rank $2$ vector bundles with given Chern classes, without resolving the uniqueness question.

The analysis of Atiyah and Rees was taken in several different directions.  On the one hand, B{\u{a}}nic{\u{a}} and Putinar studied algebraizability questions on more general affine threefolds in a series of papers culminating in  \cite{BanicaPutinar1987VectorBundlesP3}.  After explicitly giving a topological classification, which involves additional invariants beyond the $\alpha$-invariant, they construct numerous algebraic vector bundles and answer the algebraizability question in various cases.  On the other hand, the case of higher dimensional projective space has attracted considerable attention; we refer the reader to \cite{OSS} for further discussion.


\subsection{Quadratic maps and homotopy theory}
\label{ss:quadraticmaps}
In their 1964 paper on the proof of the periodicity theorem for complex vector bundles, Atiyah and Bott \cite{AtiyahBott1964Bott} begin by recalling that topological (complex) vector bundles over $S^2$ are well understood, and suggesting that if the classification of such vector bundles can be understood ``sufficiently intrinsically'', then it would yield information about vector bundles on $S^2 \times X$.  They proceed to sketch such an ``intrinsic'' description leading in two directions, the first of which, in their words, ``brings one close to the work of Grothendieck in algebraic geometry''.  They draw a lesson from their analysis, which we quote here.

\begin{quote}
	In algebraic topology the orthodox method is to replace continuous maps by simplicial approximations, and then use combinatorial methods. When the spaces involved are differentiable manifolds a powerful alternative is to approximate by differentiable maps and use differential geometric techniques. The original proof of the periodicity theorem, using Morse Theory, was of this nature. What we have done here is to use polynomial approximation and then apply algebraic techniques. In principle this method is applicable whenever the spaces involved are algebraic varieties. It would be interesting to see this philosophy exploited on other problems.
\end{quote}

The suggestion of Atiyah and Bott was taken up by several authors.  Atiyah and Bott already mention the work of Reg Wood \cite{Wood1966BanachBott} which applies ``polynomial'' techniques to the analysis of real Bott periodicity.  Around the same time, there was intense activity linking K-theory and stable homotopy theory, via the analysis of the J-homomorphism.  Bearing this in mind, in \cite{Baum1967Quadratic}, Paul Baum poses the question: ``Is there, for example, some specially simple class of polynomial maps which carries the stable homotopy of spheres?''  His analysis yields the following two conclusions:
\begin{enumerate}[noitemsep,topsep=1pt]
	\item The stable $J$-homomorphism can be interpreted as an algebraic operation which converts a linear map into a quadratic map.
	\item Any element of a $k$-stem can be represented by a quadratic map $q: R^n \to R^l$ such that $q(S^{n-1}) \subset R^l \setminus 0$.
\end{enumerate}
Baum's result is entirely existential, though he does recall the algebraic nature of various standard maps (e.g., the Hopf map $\eta: S^3 \to S^2$), but he adds: ``perhaps eventually the $k$-stem may be envisaged as a group of equivalence classes of quadratic forms''.

One could, of course, pose the same question for unstable homotopy of spheres.  Here, the situation is considerably more subtle, and the question of polynomial representatives was taken up by Reg Wood.  In his short note \cite{Wood1968Polynomial}, he considers the sphere as a real algebraic variety, and analyzes the question of whether real polynomial maps of spheres can be used to represent homotopy classes.  His response is resoundingly negative: his Theorem 2 establishes that if $n$ is a power of $2$, then every polynomial map $S^n \to S^{n-1}$ is constant; from this he deduces that if $n \geq 2r$, then every polynomial map $S^n \to S^r$ is constant.  

Wood's analysis presumably curtailed further efforts at finding polynomial representatives.  Nevertheless, Wood revisited the problem in \cite{Wood1993Quadrics} some 25 years later, with a small twist: he observes that the situation is quite different if one replaces spheres by their complexifications, i.e., he studies the varieties $\mathrm{S}^n_{\cplx}$ mentioned in the introduction.  As regards the possibility of existence of polynomial representatives in this situation, his tone is much more measured, and he poses the following question: ``What is the first element in the homotopy groups of spheres which cannot be represented by a polynomial map of quadrics?''  Though he poses this question, immediately afterwards he qualifies it with the addition: ``It would seem optimistic to expect the affine category to be rich enough to represent all homotopy classes, but we cannot rule this out at present''.

By way of comparison with the story we told earlier, one might also ask about the realization of continuous maps by {\em holomorphic} maps of quadrics.  This story falls squarely within the later analysis of Oka-type problems by Cartan, Grauert and his school.  Indeed, the variety $\mathrm{S}^n_{\cplx}$ is a homogeneous space for a complex Lie group.  Extending Cartan's early analysis, Grauert studied the question of whether, if $X$ is a Stein manifold and $G$ is a complex Lie group, any continuous map $X \to G$ is homotopic to a holomorphic map \cite{Grauert1957Holomorphe}.  Later, Ramspott showed that a similar statement held for homogeneous spaces under complex Lie groups \cite{Ramspeott1965Schnitte}; since $\mathrm{S}^m_{\cplx}$ is a homogeneous space for the complex Lie group $SO(m+1,\cplx)$, one immediately deduces the following result.

\begin{theorem}
	\label{thm:holomorphicrepresentatives}
	For any pair of integers $m,n$, every element of $\pi_n(S^m)$ can be represented by a holomorphic map $\mathrm{S}^n_{\cplx} \to \mathrm{S}^m_{\cplx}$.
\end{theorem}  

\subsection{Motivic results}
\label{ss:motivicresults}
The techniques of motivic homotopy theory, for example the representability theorem for vector bundles, Theorem~\ref{thm:vbrepresentability}, and the subsequent infrastructure of Morel's results from \cite{MField}, provide one way to explore the structure of vector bundles on affine varieties beyond the ``low-dimensional'' cases considered above.  For example, the results of \cite{MurthySwan1976VectorBundlesAffineSurfaces} and \cite{KumarMurthy1982AlgebraicCyclesAffineThreefolds} can be put into a broader context.  We refer the reader to \cite{AsokFasel2022VectorBundles} for further discussion in this direction.

Already for smooth affine varieties of dimension $\geq 4$, the story changes considerably.  Indeed, one can see that Question~\ref{question:griffiths} admits a negative answer.  The basic idea is simple: vector bundles are not entirely determined by their Chern classes in higher dimensions.

\begin{theorem}[A.-Fasel-Hopkins {\cite{AsokFaselHopkins2019Obstructions}}]
	Suppose $X$ is a smooth complex affine variety of dimension $4$, and $\mathcal{E}^{an} \to X^{an}$ is a rank $2$ complex analytic vector bundle with Chern classes $c_i^{top} \in H^{2i}(X^{an},\Z)$.  Assume the Chern classes $c_i^{top}$ of $\mathcal{E}^{an}$ are algebraic, i.e., lie in the image of the cycle class map $cl$.  
	\begin{enumerate}[noitemsep,topsep=1pt]
		\item The bundle $\mathcal{E}^{an}$ is algebraizable if and only if we may find $(c_1,c_2) \in CH^1(X) \times CH^2(X)$ with $(cl(c_1),cl(c_2)) = (c_1^{top},c_2^{top})$ such that $Sq^2c_2 + c_1 \cup c_2 = 0 \in CH^3(X)/2$.
		\item There exist smooth complex affine varieties of dimension $4$ such that $cl_1$ and $cl_2$ are bijective, and which carry a (necessarily non-trivial) rank $2$ vector bundle $\mathcal{E}^{an}$ for which the above obstruction is non-trivial; consequently $\mathcal{E}^{an}$ admits no algebraic structure. 
	\end{enumerate}
\end{theorem}

These results answer comparison questions that classical methods left open.  It is not surprising, in a sense, that Question~\ref{question:griffiths} admits a negative answer in higher dimensions since the integral version of the Hodge conjecture is false in higher dimensions.  In effect, the construction of non-trivial examples as above relies on the failure of the integral Hodge conjecture.  

In a different direction, theorems like Theorem~\ref{thmintro:cellularalgebraizability} from the introduction were suggested by some specific algebraizability results, e.g., those formulated in \cite{AFHRees}.  The collapse map ${\mathbb P}^n \to S^{2n}$ allows one to build rank $2$ topological vector bundles on ${\mathbb P}^n$ from elements in homotopy of $S^3 \cong SU(2)$.  In \cite{AFHRees}, explicit motivic homotopy classes lifting families of elements in odd homotopy of $S^3$ are constructed, giving rise to numerous non-trivial algebraic vector bundles on any smooth affine variety equivalent in $\Spc(k)$ to ${\mathbb P}^n$, e.g., the variety $\widetilde{{\mathbb P}^n}$ from the introduction.

Viewed from this perspective, the comparison problems of the introduction appear almost inevitable. They emerge as the meeting point of several mathematical developments, each of which seems naturally to call for comparison between algebraic, analytic, and topological objects.


\section{A genealogy of algebraic-analytic comparison problems}
\label{s:genealogy}
The preceding section reconstructed one route by which the mathematics of the 1950s and 1960s can be understood as leading toward the comparison problems of the introduction. Such a reconstruction is mathematically illuminating, but it also has a characteristic effect: it makes those problems appear almost inevitable. The historical question is different. Rather than asking how these mathematical developments fit together from the standpoint of the present, we ask how comparison problems of the kind studied here became mathematically meaningful to the mathematicians who first posed them. To that end, we sketch one genealogy of these problems, following several intertwined themes in the development of mathematics from the middle of the nineteenth century through the middle of the twentieth.  Our concern is not priority, but the emergence of a mathematical vocabulary in which these questions could be posed.

The importance of such an analysis was already recognized, at least in broad outline, by David Hilbert.  In the text of his 1900 ICM address, Hilbert included a long meditation on the nature of problems in mathematics, before actually stating his now famous problems.  In these meditations, he asserts \cite[pp. 437-438]{Hilbert1902MathematicalProblems}:  
\begin{quote}
	We know that every age has its own problems, which the following age either solves or casts aside as profitless and replaces by new ones...It is difficult and often impossible to judge the value of a problem correctly in advance; for the final award depends upon the gain which science obtains from the problem....It should be to us a guide post on the mazelike paths to hidden truths, and ultimately a reminder of our pleasure in the successful solution.
\end{quote}
Hilbert's framing emphasizes the historically contingent nature of mathematical problems. The problems that mathematicians regard as significant change with shifting interests, methods, and conceptual frameworks, and those transformations themselves reflect the evolving character of mathematical understanding.  

Bearing this in mind, our aim here is to narrate a genealogy of a problem posed by Jean-Pierre Serre in a S\'eminaire Bourbaki talk in May 1953, which one might view as a primordial model of the kind of question posed in the introduction.  One mechanism by which mathematical problems are routinely refashioned is almost unwitting: when stating a problem posed in an earlier historical period, we recast the problem by alternately modernizing terminology or subtly modifying the formulation to account for developments after the problem was posed.  In the interest of avoiding such subtle changes, we reproduce Serre's statement in its original form \cite{SerreFSBBKI}; fortunately, the mathematical terminology employed is similar enough to modern usage that all that is necessary is some explanation of the notation, which we provide afterwards.
\begin{center}
	\includegraphics[width = .9\textwidth]{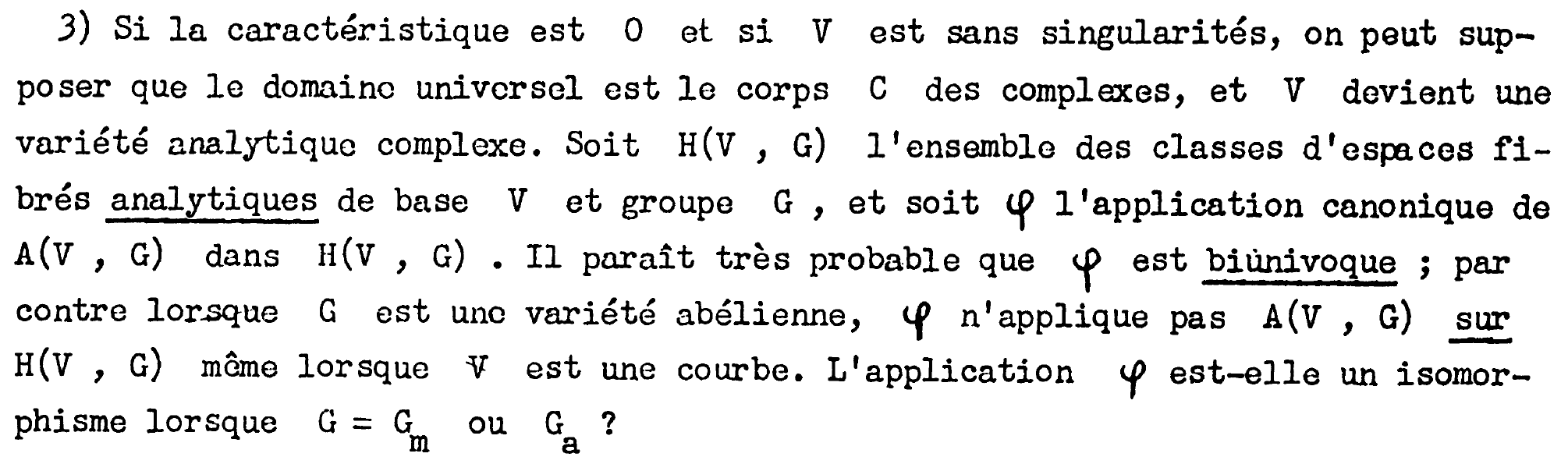}
\end{center}

Let us first clarify the problem statement.\footnote{In the interest of readability, let us give a somewhat literal translation: ``If the characteristic is $0$ and $V$ is nonsingular, we can suppose that the universal domain is the field $\cplx$ of complex numbers, and $V$ becomes a complex analytic variety.  Let $H(V,G)$ be the set of classes of analytic fiber spaces with base $V$ and group $G$, and let $\varphi$ be the canonical map from $A(V,G)$ to $H(V,G)$.  It appears very probable that $\varphi$ is bijective; by contrast, when $G$ is an abelian variety, $\varphi$ does not map $A(V,G)$ onto $H(V,G)$ even when $V$ is a curve.  Is the map $\varphi$ an isomorphism when $G = G_m$ or $G_a$?''}  The focus is evidently on complex algebraic varieties; for an explication of the notation set out here, one must revisit some earlier points in the document.  In Section 1, Serre explains that $A(V,G)$ is the set of isomorphism classes of algebraic fiber bundles with base $V$ and group $G$, explicitly a group-variety of some form; such fiber bundles are assumed locally trivial in the {\em Zariski} topology on $V$ (indeed, finer notions of local triviality in the algebraic setting had not yet been invented).  Throughout the note, Serre explicitly specifies the additional hypothesis that $V$ is complete when he views that as necessary, but that hypothesis is absent here even if it evidently carries special interest.  

Bearing all this in mind, one could ask: why would one pose this question in the first place?  We will avoid the answer ``because one can'': the problem requires enough terminology and mathematical background that a reasonable reader might require more information to indicate why one should spend time studying this problem.  Said differently, what features of the mathematical landscape supported interest in comparison of algebraic and analytic fiber bundles at this particular time?  Our goal in the subsequent discussion is not to provide a ``proper'' historical analysis, whatever that might mean, but rather to highlight a few contingent factors relevant to the problem being mathematically meaningful.  We aim to argue in support of the thesis that this problem of Serre's would not have been formulated before he did so.  Importantly, the further consideration of the problem was context dependent: Serre posed this problem in the Bourbaki seminar, a group of mathematicians with related common culture.  A much more detailed analysis of contingent factors around this problem can be found in the first author's forthcoming book (though see \cite{Asok} for a preliminary account). 

\subsection{Continuity and the intuitive nature of functions}
\label{ss:continuity}
Algebraic geometers routinely refer to the Greek heritage of the subject, perhaps implicitly in reference to conic sections (see, e.g., \cite{Dieudonne1985HistoryAG}).  However, the idea of ``algebraic geometry'' as something about ``algebraic functions'' could not exist until a suitably refined notion of function existed, which seems to have crystallized around the 17th century.  To situate our analysis of Serre's problem, let us turn back the clock a bit and analyze the role of ``algebraic'' functions within the broader investigation of the function concept.  

Using the term analytic for the time being to refer to any kind of ``definite'' formulaic regularity, the notion of analytic expression of functions is connected with Fermat and Descartes \cite[p. 36]{medvedev1991scenes}.  In the context of geometry, Descartes' {\em La Geometrie} \cite{descartes1954geometry} can perhaps be viewed as an attempt to systematically approach problems of geometry by algebraic (e.g., ``arithmetical'') means.  A fundamental shift appears in work of Newton: Newton mastered Descartes in his early years, but had a rather conflicted relationship with arithmetization, seemingly asserting the primacy of geometry over Cartesian  representation.\footnote{See \cite[Part II]{Guicciardini2011NewtonCertaintyMethod} for a detailed analysis of Newton's views on Descartes.  For a briefer discussion, relevant in the context of mathematical practice, we mention \cite[\S 4.1]{Arana2023-ARAMHO}.}  Nevertheless, Newton popularized the use of ``explicit'' descriptions of function \cite[\S 2.4]{medvedev1991scenes}, though the name ``function'' itself was seemingly first used by Leibniz \cite[p. 58]{monna1972function}.  

By the 18th century, a conception of at least a ``function of one variable'' (possibly complex) had been proposed: a function was an analytical expression in a variable, composed by use of addition, subtraction, multiplication, division, extraction of roots, trigonometrical and logarithmic operations (see \cite[p. 58]{monna1972function} and \cite[\S 2.5]{medvedev1991scenes}).  By the early 19th century, this notion of function as formula had stabilized somewhat, though the nature of such formulas begins to shift, in part because of the influence of trigonometric series.  Such considerations engendered more refined notions of continuity. At the beginning of the 19th century, there were routine attempts to prove that ``continuous'' functions were even differentiable at perhaps some isolated bad points.

In 1829, Dirichlet had written down his nowhere continuous function (the function taking the value $1$ on the rational numbers and $0$ on irrational numbers) as a limit of trigonometric functions \cite[p. 169]{dirichlet1829convergence}.   Lobachevsky writes in 1837:
\begin{quote}
	The general concept of a function requires that a function of $x$ be defined as a number given for each x and varying gradually with x. The value of the function can be given either by an analytic expression, or by a condition that provides a means of examining all numbers and choosing one of them; or finally, the dependence may exist and remain unknown.  
\end{quote}
Dirichlet gives a similar sounding definition around the same time, but explicitly restricts his attention to continuous functions \cite[\S 2.8]{medvedev1991scenes}.  This Dirichlet--Lobachevsky definition of ``function'' as a source for the modern notion is clouded by some ambiguity around the importance of continuity.  Nevertheless, by this point one could reasonably speak of ``familiar'' functions given by analytic expressions.

The notion of function underwent another shift in the latter part of the 19th century.  Riemann asserted that an explicit trigonometric series was continuous but nowhere differentiable, without providing any indication of justification.  Weierstrass began his famous note providing an explicit such function in the early 1870s.  Once again, Weierstrass's example arose from considerations involving Fourier series \cite{Weierstrass1895Continuous}.  

Complementing this construction, in his seventies, Weierstrass established results that we now package as his famous ``approximation theorem'' \cite{Weierstrass1885AnalytischeDarstellbarkeit}.  Indeed, for any $f$ that is continuous and bounded on $\real$, setting
\[ 	
F(x,k) := \frac{1}{k \sqrt{\pi}} \int_{-\infty}^{\infty} f(u) e^{-(\frac{u-x}{k})^2} du, 
\]
he showed that
\[
\lim_{k \to 0^+} F(x,k) = f(x),
\]
where, in modern terminology, convergence is uniform on compact intervals.  Importantly, the function $F(-,k)$ is entire for $k > 0$.  As consequences of this observation, two other facts follow by straightforward means: i) any continuous function on a compact interval can be approximated by polynomials (i.e., what we nowadays usually refer to as the Weierstrass approximation theorem), and ii) if $f$ is a continuous and bounded function, then $f$ can be represented (in many ways) by an infinite series of polynomials, which converges absolutely pointwise, and uniformly on intervals.  

Further examples only laid bare the disjunction between continuity and ``familiar'' functions in stark relief, leading Poincar\'e to his famous statement \cite{PoinareDefinitions}:
\begin{quote}
	Logic sometimes breeds monsters. For half a century, we’ve seen the emergence of a host of bizarre functions that seem to strive to resemble honest, useful functions as little as possible. No continuity, or continuity but no derivatives, etc. Moreover, from a logical standpoint, these strange functions are the most general; those encountered unintentionally now appear only as a special case.  They have only a tiny corner left.
\end{quote}
By the beginning of the 20th century, the work of Cantor, Baire, Banach, etc. established that Poincar\'e's statement was no rhetorical flourish: in a precise sense, most continuous functions are nowhere differentiable, thus establishing ``true'' continuity as a counter-intuitive phenomenon.  

Weierstrass's proof of his approximation theorem, as indicated above, was by construction formulaic.  Simultaneously, and frequently identified with Hilbertian formalism, the beginning of the 20th century saw a shift among conceptions of problems within pure mathematics: it became permissible to eschew {\em construction} of solutions and focus instead on the problem of existence of a solution \cite[p. 336]{monna1983algebraic}.  This shift notwithstanding, the Weierstrass approximation theorem, in the polynomial form, was given a constructive proof by Sergei Bernstein.  These results solidify the idea of polynomials as computable avatars of abstract and potentially inscrutable ``general'' continuous functions.

\subsection{Modernizing spaces: algebraic varieties among manifolds}
The notion of manifold is frequently tracked to Riemann,\footnote{For a more detailed analysis of the evolution of the manifold concept, we refer the reader to \cite{Scholz}} but for Riemann the notion was rather more of a psychological construct.  His famous essay based on his Habilitationvortrag, published as {\em The hypotheses on which geometry is based} begins by begging his readers' indulgence for making philosophical speculations about $n$-dimensional entities \cite[p. 257]{Riemann}.  Spaces under consideration arose from geometry and physics and Riemann's introduction of the manifold notion is based on measurement.  Riemann's description lays bare another implicit assumption we frequently make in describing manifolds as he mentions points or elements, long before set-theoretic conceptions were used in mathematics.

In contrast to his frequently physical motivations, Riemann's view of ``continuous'' manifolds can be read as suggesting a disjunction between the needs of mathematics and physics:
\begin{quote}
	...occasions which give rise to notions whose measurement involves the consideration of continuous manifolds are so  rarely encountered in everyday life that the location of material objects perceived through the senses, and colors, are perhaps the only simple examples of concepts whose modes of determination constitute a multi-dimensional manifold. Not until we enter the realm of higher mathematics does the need to create and develop such concepts make itself felt. 	
\end{quote}
Insofar as mathematical justifications are concerned, Riemann writes: ``Such studies have become a necessity in various parts of mathematics, notably in the treatment of multi-valued analytic functions. The lack of these studies may well be one of the main reasons why Abel's famous theorem and the contributions of Lagrange, Pfaff, and Jacobi to the general theory of differential equations have remained unfruitful for so long.'' 

The notion of manifold remained vague for quite some time; Felix Klein speaks of {\em manifoldness}.  Poincar\'e opens {\em Analysis Situs} with the assertion ``Nobody doubts nowadays that the geometry of n dimensions is a real object''.  If analysis had become the primary means of interrogating geometry, leaving geometry as secondary, Poincar\'e explicitly hopes to re-establish the primacy of geometry, a tradition going back to Newton.  He writes \cite{PoincareTrans}:
\begin{quote}
	But why, it may be said, not preserve the analytic language and replace the language of geometry, as this will have the advantage that the senses can no longer intervene. It is that the new language is more concise; it is the analogy with ordinary geometry which can create fruitful associations of ideas and suggest useful generalizations.
\end{quote}
He goes on to explain points where the analytic approach is, in his view, entirely unsuitable. 

While it is natural for us now to describe a world of embedded smooth manifolds as ``cut out from Euclidean space'' by finitely many equations, this notion does not appear until Poincar\'e's {\em Analysis Situs} in 1895; his first definition of manifold (he gives several) is phrased precisely in terms of a system of simultaneous equations and inequalities.  Importantly, for Poincar\'e the defining functions are continuous with continuous first partial derivatives and the matrix defined by the partial derivatives of the equations with respect to the given variables has maximal rank at points where the equalities are satisfied.  This is effectively our modern notion of embedded ($C^1$) differentiable manifold, and certainly one we suggest to students as {\em intuitive}.  However, Poincar\'e, building on Riemann's ideas around analytic continuation, also gives another definition in terms of overlapping patches.  

In {\em Analysis Situs}, Poincar\'e builds tools for a qualitative theory of manifolds, and he provides several justifications for such a theory.  Among the examples he lists, I emphasize two:
\begin{quote}
	The classification of algebraic curves into types rests, after Riemann, on the classification of real closed surfaces from the point of view of Analysis situs. An immediate induction shows us that the classification of algebraic surfaces and the theory of their birational transformations are intimately connected with the classification of real closed hypersurfaces in the space of five dimensions from the point of view of Analysis situs. M. Picard, in a memoir honoured by the Acad\'emie des Sciences, has already insisted on this point.
	
	Then again, in a series of memoirs in Liouville's journal, entitled: {\em Sur les courbes d\'efinies par les \'equations differentielles} I have employed the ordinary analysis situs of three dimensions in the study of differential equations. The same researches have been pursued by M. Walther Dyck. We can easily see that generalized Analysis situs will permit us to treat higher order equations in the same way, in particular, the equations of celestial mechanics.
\end{quote}
But this discussion mentions no distinguished place for algebraic varieties among manifolds more generally.  Let us note that the ``modern'' notions of topological space did not appear until Hausdorff's famous 1914 work, and the ``modern'' notion of differentiable manifold only crystallized in the work of Veblen and Whitehead from the 1920s and early 1930s, stabilizing with the work of, e.g., Whitney from this period.\footnote{Again, see the analysis of \cite{Scholz} for further discussion of the manifold concept.}

The idea of a separation between algebraic varieties and manifolds owes much to the work of Andr\'e Weil.  For example, his analogical use of the tools of the Italian school of algebraic geometry to approach the Riemann hypothesis for varieties over finite fields is laid out in a letter to Emil Artin in the early 1940s, though his correspondence with Hasse in the 1930s already uses much geometric language (see \cite[Chapter 10]{roquette2018riemann} for further discussion).  In the introduction to his {\em Foundations of algebraic geometry}, Weil writes explicitly \cite{WeilFoundations}:
\begin{quote}
	this book has arisen from the necessity of giving a firm basis to Severi's theory of correspondences on algebraic curves, especially in the case of characteristic $p \neq 0$ (in which there is no transcendental method to guarantee the correctness 
	of the results obtained by algebraic means), this being required for the solution of a long outstanding problem, the proof of the Riemann hypothesis in function-fields. 
\end{quote}
Here, Weil's reference to ``transcendental methods'' is an implicit reference to Riemann's use of complex function theory, and later topological techniques.

This theme of separation appears again when Weil points to a specific source of a fracture between algebra and geometry in his address to the 1954 International Congress of Mathematicians \cite{Weil1956Abstract}.
\begin{quote}
	The most decisive progress ever made in the theory of algebraic curves was achieved by Riemann precisely by introducing such methods. Later authors took considerable pains to obtain the same results by other means. In so doing, they were motivated, at least in part, by the fact that Riemann had given no justification for Dirichlet's principle and that it took many years to find one. 
\end{quote}
Once again, ``such methods'' refers to transcendental techniques, especially complex function theory, and it is reasonable to ask what a ``theory of algebraic curves'' might even mean before Riemann.  

Weil also self-consciously remarks about the ``extension'' of algebraic geometry to positive characteristic.
\begin{quote}
	As to denying any existence to algebraic geometry of non-zero characteristic, not merely would this, in view of recent developments, amount to denying motion; it would also deprive
	algebraic geometry of a rich and promising field of possible applications to number-theory, where one cannot do without reduction modulo $p$.
\end{quote}
Moreover, not until the late 1940s was it widely agreed that algebraic varieties exist as independent geometric entities with the new conception of Zariski topology.

To demarcate algebraic geometry as a subject might also indicate some methodological uniformity provided by the use of ``algebraic techniques'', but nothing could be further from the truth.  Mumford reminisces about the writing of his notes on algebraic geometry from the early 1960s suggesting methodological anarchy \cite{Mumford1999RedBook}:
\begin{quote}
	It may be of some interest to recall how hard it was for algebraic geometers, even knowing the phenomena of the field very well, to find a satisfactory language in which to communicate to each other. At the time these notes were
	written, the field was just emerging from a twenty-year period in which every researcher used his own definitions and terminology, in which the "foundations" of the subject had been described in at least half a dozen different mathematical
	"languages".
\end{quote}
Mumford goes on to describe work of ``the Italian school'', Zariski, Weil, and he does not even mention the work of Chevalley; highlighting Weil and his Foundations above was for narrative convenience.  

One key distinction between the Weil and Zariski approaches was Weil's insistence on considering ``abstract'' algebraic varieties, which included two distinct features.  On the one hand, this referenced techniques: ``Let us call "abstract" those methods
which, being basically algebraic, are essentially applicable to arbitrary ground fields; this includes for instance the theory of differentials of the first, second and third kinds (but of course not that of their integrals) and the greater part of the "geometric" proofs of the Italian school.''  However, it also impinged on the precise objects of study: Weil's Foundations introduced {\em abstract varieties}, motivated by analogy with the concept of abstract manifold.  Meanwhile, Zariski and those in his orbit were content to study ``classical'' projective varieties, i.e., those that correspond to embedded manifolds.  

The divergence between the ``classical'' and ``abstract'' streams lays the groundwork for rendering problems like ``characterize nonsingular algebraic varieties among manifolds'' as meaningful.  Focusing on this question is not arbitrary.  John Nash, in his address to the 1950 ICM entitled ``Algebraic approximations of manifolds'' \cite{Nash1952Algebraic}, explicitly tracing questions like this back to work of Seifert in the mid 1930s, considers exactly this kind of question in the context of algebraic varieties defined by real equations.  Nash is also conscious that ``classical'' algebraic geometry as it was conceived displayed insufficiencies: ``Since the terminology of classical algebraic geometry is inadequate for our purposes, we must make a few definitions'' and subsequently proposes a definition suitable to his immediate needs, thus isolating the notion that we now call {\em Nash} functions.  Likewise, at this point, given manifolds that are algebraically defined, one can meaningfully ask whether a map between two such objects admits a polynomial representation as well.

\subsection{Fiber bundles: examples or avatars?}
\label{ss:fiberbundles}
Serre's Bourbaki lecture makes heavy reference to a document of Weil's: a transcript of a talk given at Chicago in 1949 entitled {\em Fiber spaces in algebraic geometry} \cite{WeilI}, expanded into lecture notes during a series of lectures Weil gave in the early 1950s \cite{WeilFBCourse}.  This raises the question of the role of fiber spaces in mathematics of the day and, more precisely, how vector bundles came to be viewed as central within that domain.

In 1950, Norman Steenrod wrote {\em The topology of fibre bundles} \cite{SteenrodFB}.  The preface to this work opens with:
\begin{quote}
	The recognition of the domain of mathematics called fibre bundles took place in the period 1935-1940. The first general definitions were given by H. Whitney. His work and that of H. Hopf and E. Stiefel demonstrated the importance of the subject for the applications of topology to differential geometry...The subject has attracted general interest, for it contains some of the finest applications of topology to other fields, and gives promise of many more. It also marks a return of algebraic topology to its origin; and, after many years of introspective development, a revitalization of the subject 
	from its roots in the study of classical manifolds.  
\end{quote} 

Here, Steenrod's reference to importance for applications can be traced to several classes of examples.  First, there is Hopf's early work on the Hopf fibration, which provided a homotopically non-constant map $S^3 \to S^2$, itself a culmination of a long string of work building from Brouwer's notion of degree to more complicated questions about classification of continuous maps between spheres.  Seifert, for wholly different reasons related to the study of $3$-manifolds (a subject which had grown in interest in part from Poincar\'e and Birkhoff's work on dynamics), also began an analysis of a class of fiber spaces.

The work of Stiefel--Whitney to which Steenrod alludes constitutes the genesis of the theory of characteristic classes, which were linked to questions about linearly independent vector fields generalizing the so-called Poincar\'e--Hopf theorem.  By the end of the 1940s, there was also work, e.g., by Eilenberg, on further extension of the problem of classifying continuous maps up to homotopy using a stepwise procedure that would eventually come to be codified as obstruction theory.  

The introspective development alluded to by Steenrod is the development of Poincar\'e's {\em Analysis Situs} from Betti numbers/torsion numbers to a plethora of different homology and cohomology theories.  These different homology theories were systematized in the axiomatic approach of Eilenberg and Steenrod.  And by the early 1940s, Whitney had formulated a general notion of fiber bundle and had emphasized the role of bundles in his approach to the study of differentiable manifolds \cite{WhitneyFB}.  

Whitney's 1935 paper \cite{Whitneyspherespace} introduces what he calls plane bundles, but it would be misleading simply to identify these with modern vector bundles. Although each fiber carries the structure of a vector space, Whitney's interest is almost entirely topological, and almost immediately he replaces the bundle of planes by its associated bundle of unit spheres (\cite[\S 10]{Whitneyspherespace} explains precisely this reduction). From the standpoint of the problems Whitney wished to study, e.g., characteristic classes, the sphere bundle retained essentially all of the relevant information while having compact fibers. In other words, the linear structure on the fibers was auxiliary rather than intrinsic.

This emphasis was not unique to Whitney, nor to the United States, and one sees parallel streams of development in Europe.  Feldbau's 1939 note \cite{feldbau_classification_1939}, while considerably broadening the notion of a fiber bundle to allow what he called "essentially arbitrary" fibers, likewise concentrated on compact fibers in the applications.  Feldbau also presents a difficult case around the communication of mathematics: he was Jewish and died under deportation before the end of
World War II.  Subject to censorship under laws of the Vichy government
beginning in 1941 (and sometimes publishing under a pseudonym), it is
unclear how to assign credit to the works authored jointly by Ehresmann \cite{ehresmann_feldbau_homotopie_1941} or even solely by Ehresmann \cite{ehresmann_espaces_associes_1941}, and furthermore to assess the visibility of these works.  Such questions have been explored in depth by Audin \cite{audin_feldbau_2009}.  Questions of transmission are therefore not external to our narrative: which mathematical formulations survived and circulated inevitably shaped which problems later became mathematically meaningful, and to whom.

The end of the war did not immediately bring conceptual consolidation, even as mathematical communication gradually resumed. Even after the publication of Whitney's Michigan lectures and the parallel work of many others, the definition of a fiber bundle itself remained unsettled. Several competing formulations coexisted, each adapted to somewhat different classes of examples. The mathematical community therefore inherited a fragmented rather than standardized vocabulary, and within this evolving framework the modern notion of a vector bundle emerged only gradually.

Chern's 1946 paper introducing characteristic classes for Hermitian manifolds \cite{Chern1946} likewise proceeds through the associated sphere bundle. Although the underlying geometry is linear and Hermitian, the characteristic-class construction is formulated in terms of sphere bundles. This should not be regarded as a merely archaic presentation of a modern vector-bundle argument. Rather, it reflects the conceptual vocabulary then available: the compact sphere bundle remained the natural topological object on which transgression and obstruction-theoretic constructions could be performed. The resumption of mathematical communication after the war did not immediately standardize that vocabulary. Remarkably, even in the Cartan seminar of late 1949, Blanchard begins his discussion of fiber bundles with a ``Définition provisoire''.  The very notion of a fiber bundle was still regarded as provisional.

It is only toward the end of the 1940s that one begins to see the linear structure itself isolated as of independent interest. The terminology fibr\'e vectoriel appears in the Cartan Seminar (see, e.g., \cite[6-12]{Cartan1949GeneralitesI}).  In a terminological direction, in {\em The Topology of Fibre Bundles}, Steenrod speaks of ``linear bundles'' and ``tensor bundles'', as well as of tangent bundles, but apparently never uses the expression ``vector bundle''. In \S 12.10 he remarks that, in studying a tangent bundle, it is frequently more convenient to replace it by the associated sphere bundle, both because the fiber dimension drops by one and because the resulting fiber is compact. Vector-space fibers were certainly present throughout the book, but the linear bundle was not yet unambiguously the privileged object. From the viewpoint of much contemporary topology, it was often a device from which one extracted a principal bundle, a homogeneous-space bundle, or a compact sphere bundle better suited to obstruction theory and cohomological calculation.

While the terminological transition is important, and we will return to it momentarily, numerous results in this period lead to the transformation of fiber bundles from examples of other structures (e.g., manifolds) to objects of independent interest.  Sphere bundles, which were highlighted by Whitney, were later classified independently by Pontryagin and Steenrod, and simultaneously, Pontryagin isolated {\em vector bundles} as fiber bundles of particular interest.  Indeed, Pontryagin and Steenrod establish a classification result for sphere bundles/vector bundles by 1944, and these results are generalized to a homotopy classification result for more general fiber bundles with a given structure group.  These developments culminate in the opening chapters of Steenrod's text, where the classification of sphere bundles, linear bundles, and general bundles appears in parallel.

\begin{center}
	\includegraphics[width = .8\textwidth]{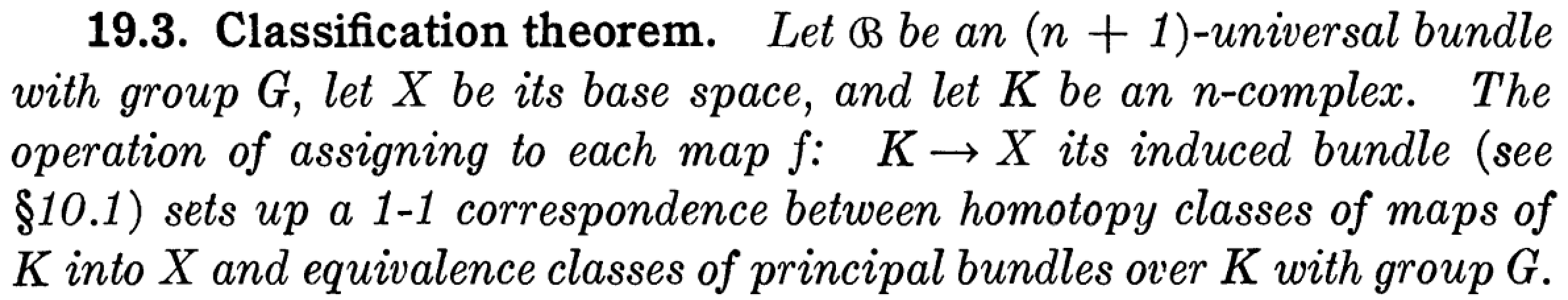}
	\vskip .5em 
	\includegraphics[width = .8\textwidth]{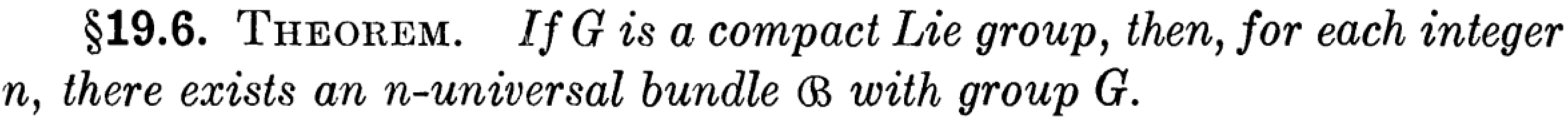}
	\vskip .5em 
	\includegraphics[width = .8\textwidth]{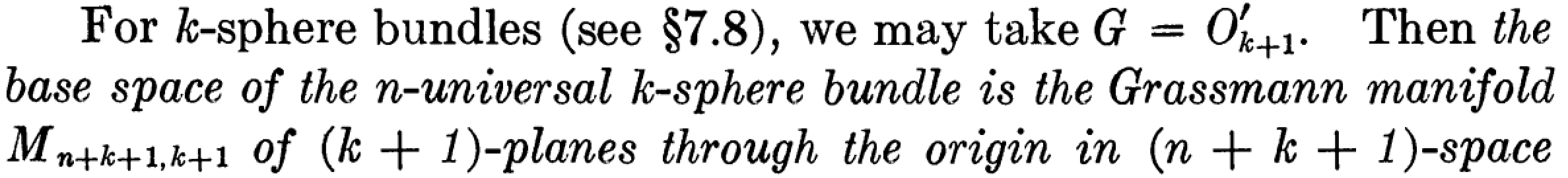}
\end{center}

Classification alone, however, did not establish vector bundles as the privileged class of fiber bundles. A further change occurred when Weil proposed that bundle theory should be systematically extended into complex analysis and algebraic geometry. In these settings the multiplicative or linear structure encoded by transition functions could not simply be discarded in favor of an associated sphere bundle. The transcript of Weil's 1949 talk begins:
\begin{quote}
	The notion of "abstract variety" makes it possible for the algebraic geometer to imitate closely the definitions and procedures which have been so fruitfully applied by modern topologists to the theory of fibre-bundles. 
\end{quote}

The chronology suggests a somewhat different genealogy from the natural claim that vector bundles arose in topology and were subsequently imported into complex analysis and algebraic geometry. Topology supplied the general concept of a fiber bundle, but did not, however, immediately establish vector bundles as the privileged class of such objects.  Instead, that distinction seems to have been driven by Weil’s attempt to transplant the fiber-bundle formalism into complex analysis and algebraic geometry: Weil viewed fiber bundles as a unifying theme in mathematics.  In the run-up to the 1950 ICM, Weil writes to Cartan \cite[p. 312]{CartanWeil}, imploring him to highlight the use of fiber bundles in complex analysis:  
\begin{quote}
	Once we get used to seeing fiber spaces in these questions, we quickly become convinced that they appear everywhere (or at least ``almost everywhere'') and that we gain enormously...It seems certain to me that almost all the problems where it is a question of gluing together local data to produce global data are of ``fiber space'' nature.
\end{quote}
Weil was speaking specifically of the Cousin problems, reinterpreting them as questions about the gluing data defining analytic bundles. Meanwhile, Cartan was reworking Leray's sheaf theory into a more usable local-to-global formalism. Their correspondence in August 1950 centers on analytic fiber bundles with group $\cplx^\times$, objects that would soon be described in the language of holomorphic line bundles, although Cartan and Weil themselves still formulate the problem in terms of bundle structures and Cousin data. In a postscript dated 5 August 1950, Cartan writes \cite[p.~315]{CartanWeil}:
\begin{quote}
	I wrote this letter last night.  This morning, I know how to show that a fiber bundle $E$ with group $\cplx^{\times}$ and base $B$ a domain of holomorphy is topologically trivial, it is analytically trivial.  The problem rests entirely on knowing that to every fiber bundle structure corresponds a Cousin datum that gives rise to it.
\end{quote}
The subsequent correspondence connects these questions with divisors, Kodaira's work, and Lefschetz's results on algebraic degree-two cohomology classes on K\"ahler manifolds. In this analytic setting, the structure discarded in passing to a sphere bundle had become indispensable.

\subsection{Algebraicity and continuity}
\label{ss:algvscont}
The reinterpretation of analytic questions in terms of fiber bundles quickly merged with another stream of ideas: the analytic theory of harmonic forms and their relationship to algebraic cycles. The relevant results emerged from several interacting mathematical traditions, especially the analytic tradition associated with H. Weyl and later W. V. D. Hodge.  One theme we highlight here is another attempt to distinguish algebraic varieties from manifolds more generally, in this case through the structure of their homology classes.

In his contribution to the 1950 ICM, W.V.D. Hodge wrote the following, which we reproduce in its entirety \cite[p. 184]{Hodge1952ICM}. 
\begin{center}
	\includegraphics[width = .8\textwidth]{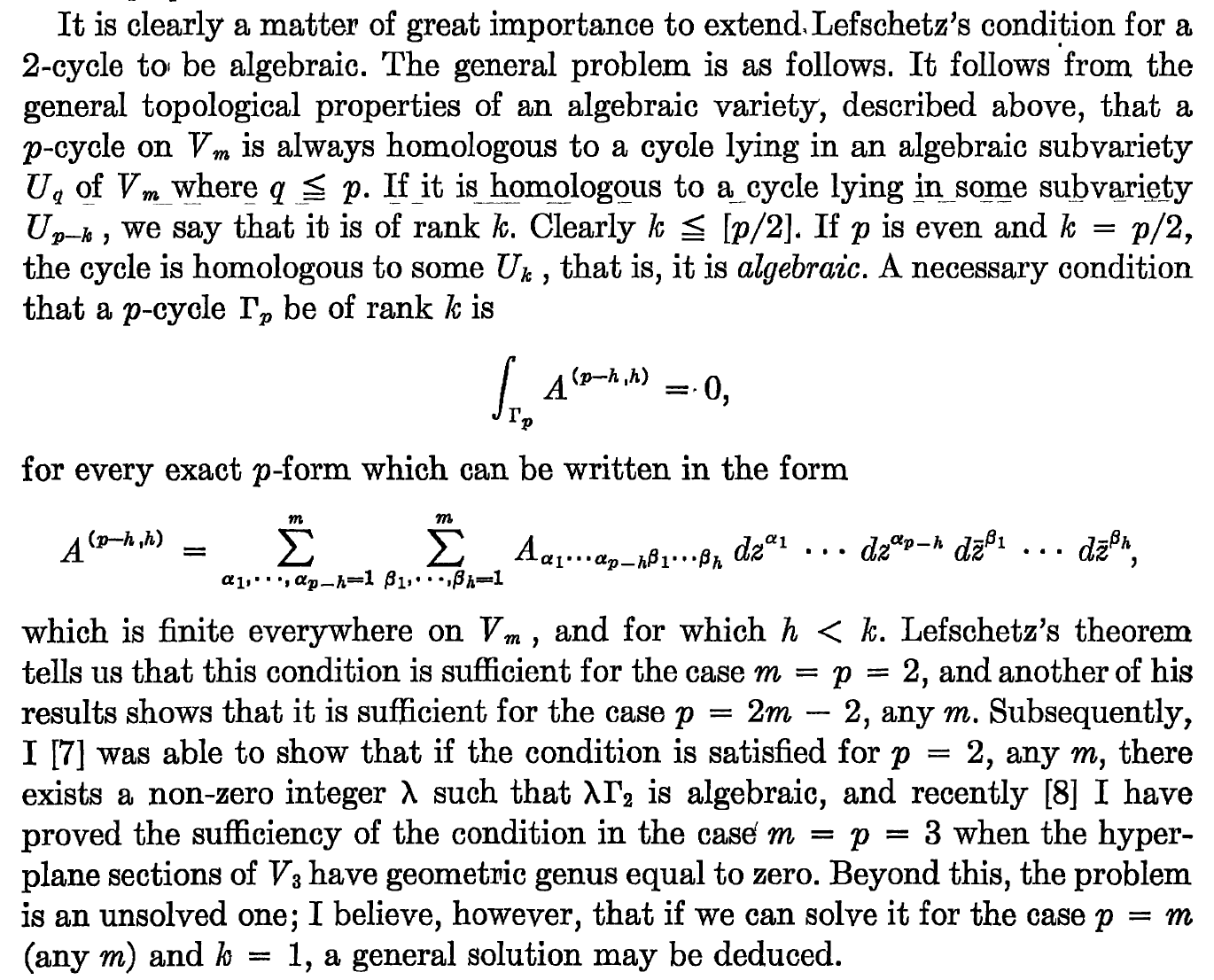}
\end{center}
What is written here has since become known as the Hodge conjecture, but only after some significant reconceptualization and clarification (one should keep in mind \cite{AtiyahHirzebruch1962Analytic} and later \cite{Grothendieck1969Hodge}); we highlight both Hodge's characterization of the problem as important, and the formulation in terms of homology.  
The Cartan--Weil correspondence shows that these questions were not merely parallel developments but were already being discussed in essentially the same language (see the letter of August 8, 1950 from Weil to Cartan \cite[p. 316]{CartanWeil}, just before the 1950 ICM).    

Hodge's formulation, as well as the setup described by Weil, is built on the analysis of harmonic representatives of cohomology classes on Riemannian manifolds from the 1930s extended to the complex analytic setting.  As recounted in various places, Hodge's original proof from the mid 1930s had a defect that was later rectified by ideas of H. Weyl.  These ideas were further clarified and extended in K. Kodaira's work from the late 1940s: see the paper on ``harmonic fields on Riemannian manifolds'' \cite{Kodaira1949HarmonicFields}, which builds on Weyl's work in the case of Euclidean space.  In 1949, around the time this paper was published, Weyl invited Kodaira to Princeton, where discussions with Weil and D.C. Spencer rapidly transformed these analytic ideas into a theory relating divisors, fiber bundles, and cohomology.

In the same letter in the Cartan--Weil correspondence mentioned above in August 1950, Weil formulates several conjectures linking analytic $\cplx^*$-bundles and divisors; in the analytic setting, Weil views the corresponding data as tantamount to the specification of ``Cousin data'' as discussed in the previous section.  By early September 1950 \cite[p. 318]{CartanWeil}, Weil states that he and Kodaira have demonstrated one of these conjectures, which amounts to a reformulation of the result of Lefschetz.  What is striking is not merely the result but the formulation:
\begin{quote}
	On a complex manifold $V$, let $u$ be a cohomology class of dimension $2$ (integer coefficients). For $u$ to be the invariant of a complex analytic fiber bundle with group $\cplx^*$, it is necessary and sufficient that $u$ be homologous to a differential form everywhere locally of the form $\sum f_{ij} dz_i d\bar{z}_j$; if $V$ is compact K\"ahler, it is necessary and sufficient for $u$ to be such an invariant that the harmonic form homologous to $u$ be of the type in question.
\end{quote}

Both the conjecture above and the other one to which we alluded make an appearance in a Seminaire Bourbaki talk by Cartan from December 1950 \cite[\S 8]{Cartan1952Bourbaki34}.  If $B$ is a compact analytic variety, Cartan observes that every analytic $\cplx^{\times}$-bundle gives rise to a class in $H^2(B,\Z)$.  He isolates two comparison problems: which elements of $H^2(B,\Z)$ are so-obtained?  Which of those correspond to a divisor?  He states that Kodaira and Weil have answered both questions in the case where $B$ is a compact complex algebraic variety that can be embedded in complex projective space.  The work which Cartan mentions appears in a paper of Kodaira submitted in December 1950 which establishes a Riemann--Roch theorem for complex analytic surfaces \cite{Kodaira1951RiemannRoch}.   In particular, this paper lays out explicitly what we now refer to as the divisor--line bundle correspondence, and attributes the results to Weil.

A tremendous shift begins to appear in the literature over the next few years: the reworking of Leray's sheaf theory by Cartan and his seminar is progressively adopted and deployed.  One trace of this adoption can be seen in the explorations of the divisor--line bundle correspondence laid out by Kodaira in joint work with Spencer announced in early 1953 \cite{KodairaSheaf, KodairaSpencer, KodairaSpencerII}.  These results recapitulate the ``harmonic'' perspective of Hodge and Weil, merging it with Weil's ``local-to-global'' philosophy, framing the discussion in terms of cohomology of sheaves.  

Around the same time, Cartan formulates what are now known as Theorems A and B, describing a theory of {\em coherent} sheaves on analytic manifolds.  Along the way, he interprets the additive Cousin problem for Stein manifolds in sheaf-theoretic terms as a consequence.  The definitive treatment appears in a conference proceedings volume from 1953 \cite{CartanBruxelles}, but more preliminary expositions appear in the Cartan seminar in the years 1951-52 (e.g., Cartan remarks in his collected works that cohomological formulation of ``Theorem B'' was not in the seminar expose).  Soon thereafter, in 1953-54, joint work of Cartan--Serre establishes a finiteness theorem for coherent analytic sheaves on compact complex analytic spaces.  

\subsection{Conclusion}
What we see in the development above is the interweaving of numerous streams of mathematics, all of which converge in Serre's Bourbaki expose: e.g., \cite[\S 2 Exemple 1]{SerreFSBBKI}.  Serre recalls the notions of divisors and linear equivalence, and carefully explains the link between divisors and fiber spaces with structure group the multiplicative group, observing that the two notions are classified by isomorphic groups.  Serre formulates explicit comparison questions between algebraic and analytic fiber bundles, including the one stated at the beginning of this section.  

That Serre's comparison question now appears natural relies on the placement of these results within a particular context.  The comparison problem was therefore not simply an internal development of algebraic geometry. It became mathematically meaningful only after a convergence of several previously independent traditions: the topological theory of fiber bundles, the analytic theory of Cousin data, Hodge's reinterpretation of cohomology classes, and Cartan's sheaf-theoretic local-to-global formalism. Serre's question emerged at the intersection of these developments.

Nevertheless, this problem of Serre's seems to have attracted considerably less attention than others he has posed.  One potential source for this discrepancy can perhaps be found in the notes on the Bourbaki expose in his collected works.  Here, Serre mentions \cite{Serre1986OeuvresI} that the comparison problem between algebraic and holomorphic fiber bundles has ``almost completely been resolved'', referencing his paper comparing sheaf cohomology on projective algebraic varieties, i.e., that which is now referred to as GAGA \cite{SerreGAGA}.  

Serre's landmark results in \cite{SerreGAGA} establish that, for $X$ projective over $\cplx$ and $G$ equal to $\mathbb{G}_m$, $\mathbb{G}_a$, or $GL_n$, the comparison map from algebraic $G$-bundles to analytic $G$-bundles is bijective.  While the statement that the comparison question has ``almost completely been resolved'' sounds rather definitive, Serre explicitly leaves open the case where \(V\) is a nonsingular complex affine variety. Griffiths' analysis \cite{Griffiths1972FunctionTheoryIA,Griffiths1972FunctionTheoryIB} suggests that this omission is more than a technical gap. The affine case, which escaped the scope of GAGA, reconnects the comparison problem with the Hodge-theoretic questions from which it first arose. The modern comparison problem thus returns, in a new guise, to its original conceptual origins. Understanding the comparison problem therefore requires understanding not only its solutions, but also the history that made those solutions worth seeking.





\begin{footnotesize}
\bibliographystyle{alpha}
\bibliography{SAGRI}

\newcommand{\etalchar}[1]{$^{#1}$}
\begin{thebibliography}{BCK{\etalchar{+}}66}

\bibitem[AB64]{AtiyahBott1964Bott}
M.F. Atiyah and R.~Bott.
\newblock On the periodicity theorem for complex vector bundles.
\newblock {\em Acta Mathematica}, 112:229--247, 1964.

\bibitem[AB23]{Arana2023-ARAMHO}
A.~Arana and H.~Burnett.
\newblock Mathematical hygiene.
\newblock {\em Synthese}, 202(110):1--28, 2023.

\bibitem[ABH23]{ABHFreudenthal}
A.~Asok, T.~Bachmann, and M.J. Hopkins.
\newblock On $\mathbb{P}^1$-stabilization in unstable motivic homotopy theory.
\newblock {\em {\em To appear} Ann. of Math.}, 2023.

\bibitem[ADF17]{ADF}
A.~Asok, B.~Doran, and J.~Fasel.
\newblock Smooth models of motivic spheres and the clutching construction.
\newblock {\em Int. Math. Res. Not. IMRN}, (6):1890--1925, 2017.

\bibitem[AE17]{AntieauElmanto}
B.~Antieau and E.~Elmanto.
\newblock A primer for unstable motivic homotopy theory.
\newblock In {\em Surveys on recent developments in algebraic geometry},
  volume~95 of {\em Proc. Sympos. Pure Math.}, pages 305--370. Amer. Math.
  Soc., Providence, RI, 2017.

\bibitem[AF59]{AndreottiFrankel1959}
A.~Andreotti and T.~Frankel.
\newblock The lefschetz theorem on hyperplane sections.
\newblock {\em Annals of Mathematics}, 69(3):713--717, 1959.

\bibitem[AF23]{AsokFasel2022VectorBundles}
A.~Asok and J.~Fasel.
\newblock Vector bundles on algebraic varieties.
\newblock In {\em Proceedings of the International Congress of Mathematicians
  2022}, volume~4, pages 2146--2170, Berlin, 2023. EMS Press.

\bibitem[AFH19]{AsokFaselHopkins2019Obstructions}
A.~Asok, J.~Fasel, and M.~J. Hopkins.
\newblock Obstructions to algebraizing topological vector bundles.
\newblock {\em Forum of Mathematics, Sigma}, 7:e6, 2019.

\bibitem[AFH22]{AFHLocalization}
A.~Asok, J.~Fasel, and M.~J. Hopkins.
\newblock Localization and nilpotent spaces in {$\mathbb A^1$}-homotopy theory.
\newblock {\em Compos. Math.}, 158(3):654--720, 2022.

\bibitem[AFH25]{AFHRees}
A.~Asok, J.~Fasel, and M.~J. Hopkins.
\newblock Algebraic vector bundles and {$p$}-local {$\Bbb A^1$}-homotopy
  theory.
\newblock {\em Ann. Sci. \'{E}c. Norm. Sup\'{e}r. (4)}, 58(5):1155--1178, 2025.

\bibitem[AH62]{AtiyahHirzebruch1962Analytic}
M.F. Atiyah and F.~Hirzebruch.
\newblock Analytic cycles on complex manifolds.
\newblock {\em Topology}, 1:25--45, 1962.

\bibitem[AHW17]{AHWI}
A.~Asok, M.~Hoyois, and M.~Wendt.
\newblock Affine representability results in {$\mathbb A^1$}-homotopy theory,
  {I}: vector bundles.
\newblock {\em Duke Math. J.}, 166(10):1923--1953, 2017.

\bibitem[AHW18]{AHWII}
A.~Asok, M.~Hoyois, and M.~Wendt.
\newblock Affine representability results in {$\mathbb A^1$}-homotopy theory,
  {II}: {P}rincipal bundles and homogeneous spaces.
\newblock {\em Geom. Topol.}, 22(2):1181--1225, 2018.

\bibitem[AHW19]{AHWOctonion}
A.~Asok, M.~Hoyois, and M.~Wendt.
\newblock Generically split octonion algebras and {$\mathbb A^1$}-homotopy
  theory.
\newblock {\em Algebra Number Theory}, 13(3):695--747, 2019.

\bibitem[AHW20]{AHWIII}
A.~Asok, M.~Hoyois, and M.~Wendt.
\newblock Affine representability results in {$\mathbb A^1$}-homotopy theory
  {III}: finite fields and complements.
\newblock {\em Algebr. Geom.}, 7(5):634--644, 2020.

\bibitem[AR76]{AtiyahRees1976VectorBundlesP3}
M.~F. Atiyah and E.~Rees.
\newblock Vector bundles on projective 3-space.
\newblock {\em Inventiones Mathematicae}, 35(2):131--153, 1976.

\bibitem[Aso22]{AQuadrics}
A.~Asok.
\newblock Affine representability of quadrics revisited.
\newblock {\em J. Algebra}, 608:37--51, 2022.

\bibitem[Aso24]{Asok}
A.~Asok.
\newblock Constructing projective modules, 2024.
\newblock \url{https://arxiv.org/abs/2412.05250}.

\bibitem[Aud09]{audin_feldbau_2009}
M.~Audin.
\newblock Publier sous l'occupation i. autour du cas de jacques feldbau et de
  l'acad{\'e}mie des sciences.
\newblock {\em Revue d'histoire des math{\'e}matiques}, 15(1):7--57, 2009.
\newblock [Publishing during German occupation of France I. The case of Jacques
  Feldbau and the Acad{\'e}mie des sciences].

\bibitem[AWW17]{AWW}
A.~Asok, K.~Wickelgren, and T.B. Williams.
\newblock The simplicial suspension sequence in {$\mathbb{A}^1$}-homotopy.
\newblock {\em Geom. Topol.}, 21(4):2093--2160, 2017.

\bibitem[Bau67]{Baum1967Quadratic}
P.F. Baum.
\newblock Quadratic maps and stable homotopy groups of spheres.
\newblock {\em Illinois Journal of Mathematics}, 11(4):586--595, 1967.

\bibitem[BC70]{BousfieldCurtis1970}
A.~K. Bousfield and E.~B. Curtis.
\newblock A spectral sequence for the homotopy of nice spaces.
\newblock {\em Transactions of the American Mathematical Society},
  151:457--479, 1970.

\bibitem[BCK{\etalchar{+}}66]{BousfieldEtAl1966}
A.~K. Bousfield, E.~B. Curtis, D.~M. Kan, D.~G. Quillen, D.~L. Rector, and
  J.~W. Schlesinger.
\newblock The mod-{$p$} lower central series and the adams spectral sequence.
\newblock {\em Topology}, 5(4):331--342, 1966.

\bibitem[BCM78]{BenderskyCurtisMiller1978}
M.~Bendersky, E.~B. Curtis, and H.~R. Miller.
\newblock The unstable adams spectral sequence for generalized homology.
\newblock {\em Topology}, 17(3):229--248, 1978.

\bibitem[Bec03]{Beck1967Triples}
J.M. Beck.
\newblock Triples, algebras and cohomology.
\newblock {\em Reprints in Theory and Applications of Categories}, (2):1--59,
  2003.
\newblock Originally published as Ph.D. thesis, Columbia University, 1967.

\bibitem[BEM25]{BachmannEngelmannMattis2025Monadic}
T.~Bachmann, A.~Engelmann, and K.~Mattis.
\newblock Monadic resolutions for generalized spaces.
\newblock {\em arXiv preprint}, 2025.

\bibitem[BK72]{BousfieldKan1972}
A.~K. Bousfield and D.~M. Kan.
\newblock The homotopy spectral sequence of a space with coefficients in a
  ring.
\newblock {\em Topology}, 11(1):79--106, 1972.

\bibitem[Blo86]{Bloch1986}
S.~Bloch.
\newblock Algebraic cycles and higher {$K$}-theory.
\newblock {\em Advances in Mathematics}, 61(3):267--304, 1986.

\bibitem[Blo94]{Bloch1994}
S.~Bloch.
\newblock The moving lemma for higher chow groups.
\newblock {\em Journal of Algebraic Geometry}, 3(3):537--568, 1994.

\bibitem[Bot59]{Bott1959Lefschetz}
R.~Bott.
\newblock On a theorem of {Lefschetz}.
\newblock {\em Michigan Mathematical Journal}, 6(3):211--216, 1959.

\bibitem[BP87]{BanicaPutinar1987VectorBundlesP3}
C.~B{\u{a}}nic{\u{a}} and M.~Putinar.
\newblock On complex vector bundles on projective threefolds.
\newblock {\em Inventiones Mathematicae}, 88:427--438, 1987.

\bibitem[Car50]{Cartan1949GeneralitesI}
H.~Cartan.
\newblock G{\'e}n{\'e}ralit{\'e}s sur les espaces fibr{\'e}s, i.
\newblock {\em S{\'e}minaire Henri Cartan}, 2:1--13, 1949--1950.
\newblock Expos{\'e} no.~6.

\bibitem[Car52]{Cartan1952Bourbaki34}
H.~Cartan.
\newblock Espaces fibr\'es analytiques complexes.
\newblock In {\em S\'eminaire Bourbaki : ann\'ees 1948/49--1949/50--1950/51,
  expos\'es 1--49}, number~1 in S\'eminaire Bourbaki, pages 281--290.
  Soci\'et\'e Math\'ematique de France, 1952.
\newblock Expos\'e no.~34, d\'ecembre 1950.

\bibitem[Car53]{CartanBruxelles}
H.~Cartan.
\newblock Vari\'{e}t\'{e}s analytiques complexes et cohomologie.
\newblock In {\em Colloque sur les fonctions de plusieurs variables, tenu \`a
  {B}ruxelles, 1953}, pages 41--55. Georges Thone, Li\`ege, 1953.

\bibitem[CG75]{CornalbaGriffiths1975Analytic}
M.~Cornalba and P.~Griffiths.
\newblock Analytic cycles and vector bundles on non-compact algebraic
  varieties.
\newblock {\em Inventiones Mathematicae}, 28(1):1--106, 1975.

\bibitem[Che46]{Chern1946}
S.-S. Chern.
\newblock Characteristic classes of hermitian manifolds.
\newblock {\em Annals of Mathematics}, 47(1):85--121, 1946.

\bibitem[CW11]{CartanWeil}
H.~Cartan and A.~Weil.
\newblock {\em Correspondance entre {H}enri {C}artan et {A}ndr\'{e} {W}eil
  (1928--1991)}, volume~6 of {\em Documents Math\'{e}matiques (Paris)
  [Mathematical Documents (Paris)]}.
\newblock Soci\'{e}t\'{e} Math\'{e}matique de France, Paris, 2011.
\newblock Edited by Mich\`ele Audin.

\bibitem[Des54]{descartes1954geometry}
R.~Descartes.
\newblock {\em The Geometry of Ren{\'e} Descartes: With a Facsimile of the
  First Edition}.
\newblock Dover Publications, Mineola, NY, 1954.
\newblock Facsimile edition with English translation from the French and Latin.

\bibitem[DF96]{Farjoun}
E.~Dror-Farjoun.
\newblock {\em Cellular spaces, null spaces and homotopy localization}, volume
  1622 of {\em Lecture Notes in Mathematics}.
\newblock Springer-Verlag, Berlin, 1996.

\bibitem[Die85]{Dieudonne1985HistoryAG}
J.~Dieudonn{\'e}.
\newblock {\em History Algebraic Geometry}.
\newblock Chapman \& Hall/CRC, New York, 1st edition, 1985.

\bibitem[Dir29]{dirichlet1829convergence}
P.~G.~L. Dirichlet.
\newblock Sur la convergence des séries trigonométriques qui servent à
  représenter une fonction arbitraire entre des limites données.
\newblock {\em Journal für die reine und angewandte Mathematik}, 4:157--169,
  1829.
\newblock In French.

\bibitem[DT58]{DoldThom1958}
A.~Dold and R.~Thom.
\newblock Quasifaserungen und unendliche symmetrische produkte.
\newblock {\em Annals of Mathematics}, 67(2):239--281, 1958.

\bibitem[EF41]{ehresmann_feldbau_homotopie_1941}
C.~Ehresmann and J.~Feldbau.
\newblock Sur les propri{\'e}t{\'e}s d'homotopie des espaces fibr{\'e}s.
\newblock {\em Comptes Rendus de l'Acad{\'e}mie des Sciences de Paris},
  212:945--948, 1941.

\bibitem[Ehr41]{ehresmann_espaces_associes_1941}
C.~Ehresmann.
\newblock Espaces fibr{\'e}s associ{\'e}s.
\newblock {\em Comptes Rendus de l'Acad{\'e}mie des Sciences de Paris},
  213:762--764, 1941.

\bibitem[Fel39]{feldbau_classification_1939}
J.~Feldbau.
\newblock Sur la classification des espaces fibr{\'e}s.
\newblock {\em Comptes Rendus de l'Acad{\'e}mie des Sciences de Paris},
  208:1621--1623, 1939.

\bibitem[God58]{Godement1958Topologie}
R.~Godement.
\newblock {\em Topologie alg\'ebrique et th\'eorie des faisceaux}, volume 1252
  of {\em Actualit\'es Scientifiques et Industrielles}.
\newblock Hermann, Paris, 1958.

\bibitem[Gra57]{Grauert1957Holomorphe}
H.~Grauert.
\newblock Holomorphe funktionen mit werten in komplexen lieschen gruppen.
\newblock {\em Math. Ann.}, 133:450--472, 1957.

\bibitem[Gra58]{Grauert1958Analytische}
H.~Grauert.
\newblock Analytische faserungen \"uber holomorph-vollst\"andigen r\"aumen.
\newblock {\em Mathematische Annalen}, 135:263--273, 1958.

\bibitem[Gri72a]{Griffiths1972FunctionTheoryIA}
P.A. Griffiths.
\newblock Function {T}heory of {F}inite {O}rder on {A}lgebraic {V}arieties.
  i(a).
\newblock {\em Journal of Differential Geometry}, 6:285--306, 1972.

\bibitem[Gri72b]{Griffiths1972FunctionTheoryIB}
P.A. Griffiths.
\newblock Function theory of finite order on algebraic varieties. i(b).
\newblock {\em Journal of Differential Geometry}, 7(1--2):45--66, 1972.

\bibitem[Gro58]{Grothendieck1958Chern}
A.~Grothendieck.
\newblock La th\'eorie des classes de {C}hern.
\newblock {\em Bulletin de la Soci\'et\'e Math\'ematique de France},
  86:137--154, 1958.

\bibitem[Gro60]{Grothendieck1960DescenteI}
A.~Grothendieck.
\newblock Technique de descente et th{\'e}or{\`e}mes d'existence en
  g{\'e}om{\'e}trie alg{\'e}brique. i. g{\'e}n{\'e}ralit{\'e}s. descente par
  morphismes fid{\`e}lement plats.
\newblock In {\em S{\'e}minaire Bourbaki : ann{\'e}es 1958/59--1959/60,
  Expos{\'e}s 169--204}, volume~5 of {\em S{\'e}minaire Bourbaki}, pages
  299--327. Soci{\'e}t{\'e} Math{\'e}matique de France, 1960.
\newblock Expos{\'e} no. 190.

\bibitem[Gro69]{Grothendieck1969Hodge}
A.~Grothendieck.
\newblock Hodge's general conjecture is false for trivial reasons.
\newblock {\em Topology}, 8:299--303, 1969.

\bibitem[Gui11]{Guicciardini2011NewtonCertaintyMethod}
N.~Guicciardini.
\newblock {\em Isaac Newton on Mathematical Certainty and Method}.
\newblock MIT Press, Cambridge, MA, 2011.
\newblock Paperback edition; originally published 2009.

\bibitem[Hil02]{Hilbert1902MathematicalProblems}
D.~Hilbert.
\newblock Mathematical problems.
\newblock {\em Bulletin of the American Mathematical Society}, 8(10):437--479,
  1902.

\bibitem[Hod52]{Hodge1952ICM}
W.~V.~D. Hodge.
\newblock The topological invariants of algebraic varieties.
\newblock In {\em Proceedings of the International Congress of Mathematicians,
  Cambridge, Massachusetts, August 30--September 6, 1950}, volume~1, pages
  182--192, Providence, RI, 1952. American Mathematical Society.

\bibitem[Hoy17]{Hoyois2017SixOperations}
M.~Hoyois.
\newblock The six operations in equivariant motivic homotopy theory.
\newblock {\em Adv. Math.}, 305:197--279, 2017.

\bibitem[KM82]{KumarMurthy1982AlgebraicCyclesAffineThreefolds}
N.~Mohan Kumar and M.~Pavaman Murthy.
\newblock Algebraic cycles and vector bundles over affine three-folds.
\newblock {\em Annals of Mathematics}, 116(3):579--591, November 1982.

\bibitem[Kod49]{Kodaira1949HarmonicFields}
K.~Kodaira.
\newblock Harmonic fields in riemannian manifolds (generalized potential
  theory).
\newblock {\em Annals of Mathematics}, 50(3):587--665, 1949.

\bibitem[Kod51]{Kodaira1951RiemannRoch}
K.~Kodaira.
\newblock The theorem of riemann--roch on compact analytic surfaces.
\newblock {\em American Journal of Mathematics}, 73(4):813--875, 1951.

\bibitem[Kod53]{KodairaSheaf}
K.~Kodaira.
\newblock On cohomology groups of compact analytic varieties with coefficients
  in some analytic faisceaux.
\newblock {\em Proc. Nat. Acad. Sci. U.S.A.}, 39:865--868, 1953.

\bibitem[KS53a]{KodairaSpencerII}
K.~Kodaira and D.~C. Spencer.
\newblock Divisor class groups on algebraic varieties.
\newblock {\em Proc. Nat. Acad. Sci. U.S.A.}, 39:872--877, 1953.

\bibitem[KS53b]{KodairaSpencer}
K.~Kodaira and D.~C. Spencer.
\newblock Groups of complex line bundles over compact {K}\"{a}hler varieties.
\newblock {\em Proc. Nat. Acad. Sci. U.S.A.}, 39:868--872, 1953.

\bibitem[Lur09]{LurieHTT}
J.~Lurie.
\newblock {\em Higher Topos Theory}, volume 170 of {\em Annals of Mathematics
  Studies}.
\newblock Princeton University Press, Princeton, NJ, 2009.

\bibitem[Lur17]{LurieHA}
J.~Lurie.
\newblock {\em Higher Algebra}.
\newblock 2017.
\newblock Available at \url{https://www.math.ias.edu/~lurie/papers/HA.pdf}.

\bibitem[Man68]{Manin1968}
Yu.~I. Manin.
\newblock Correspondences, motifs and monoidal transformations.
\newblock {\em Mathematics of the USSR-Sbornik}, 6(4):439--470, 1968.

\bibitem[Mat24a]{Mattis2024ArithmeticFracture}
K.~Mattis.
\newblock Unstable arithmetic fracture squares in $\infty$-topoi.
\newblock {\em arXiv preprint}, 2024.

\bibitem[Mat24b]{Mattis2024UnstablePCompletion}
K.~Mattis.
\newblock Unstable $p$-completion in motivic homotopy theory.
\newblock {\em arXiv preprint}, 2024.

\bibitem[Med91]{medvedev1991scenes}
F.~A. Medvedev.
\newblock {\em Scenes from the History of Real Functions}, volume~7 of {\em
  Science Networks. Historical Studies}.
\newblock Birkh{\"a}user Basel, Basel; Boston; Berlin, 1991.

\bibitem[Meh81]{Mehrtens}
H.~Mehrtens.
\newblock Socia1 history of mathematics.
\newblock In {\em Social history of nineteenth century mathematics. ({Papers}
  from a {Workshop}, {Technical} {University} {Berlin}, {July} 5-8, 1979)},
  pages pp. 257--280. 1981.

\bibitem[Mil63]{Milnor1963}
J.~Milnor.
\newblock {\em Morse Theory}, volume~51 of {\em Annals of Mathematics Studies}.
\newblock Princeton University Press, Princeton, NJ, 1963.

\bibitem[Mil12]{MilneMotivesOnline}
J.S. Milne.
\newblock Motives: Grothendieck's dream.
\newblock \url{https://www.jmilne.org/math/xnotes/mot.html}, 2012.
\newblock Version 2.04, April 24, 2012. Accessed July 20, 2026.

\bibitem[Mon72]{monna1972function}
A.F. Monna.
\newblock The concept of function in the 19th and 20th centuries, in particular
  with regard to the discussions between baire, borel and lebesgue.
\newblock {\em Archive for History of Exact Sciences}, 9(1):57--84, 1972.

\bibitem[Mon83]{monna1983algebraic}
A.~F. Monna.
\newblock Algebraic and set-theoretic aspects of the evolution of mathematics.
\newblock {\em Indagationes Mathematicae (Proceedings)}, 86(3):329--341, 1983.
\newblock Proceedings A, September 26, 1983.

\bibitem[Mor06]{MICM}
F.~Morel.
\newblock {$\Bbb A^1$}-algebraic topology.
\newblock In {\em International {C}ongress of {M}athematicians. {V}ol. {II}},
  pages 1035--1059. Eur. Math. Soc., Z\"{u}rich, 2006.

\bibitem[Mor12]{MField}
F.~Morel.
\newblock {\em {$\mathbb A^1$}-algebraic topology over a field}, volume 2052 of
  {\em Lecture Notes in Mathematics}.
\newblock Springer, Heidelberg, 2012.

\bibitem[MS74]{MilnorStasheff1974}
J.~Milnor and J.D. Stasheff.
\newblock {\em Characteristic Classes}, volume~76 of {\em Annals of Mathematics
  Studies}.
\newblock Princeton University Press, Princeton, NJ, 1974.

\bibitem[MS76]{MurthySwan1976VectorBundlesAffineSurfaces}
M.~Pavaman Murthy and R.~G. Swan.
\newblock Vector bundles over affine surfaces.
\newblock {\em Inventiones Mathematicae}, 36:125--165, December 1976.

\bibitem[Mum99]{Mumford1999RedBook}
D.~Mumford.
\newblock {\em The Red Book of Varieties and Schemes: Includes the Michigan
  Lectures (1974) on Curves and Their Jacobians}, volume 1358 of {\em Lecture
  Notes in Mathematics}.
\newblock Springer-Verlag, Berlin, 2 edition, 1999.

\bibitem[MV99]{MV}
F.~Morel and V.~Voevodsky.
\newblock $\mathbb{A}^1$-homotopy theory of schemes.
\newblock {\em Publications Mathématiques de l'Institut des Hautes Études
  Scientifiques}, 90(1):45--143, 1999.

\bibitem[MVW06]{MVW}
C.~Mazza, V.~Voevodsky, and C.~Weibel.
\newblock {\em Lecture notes on motivic cohomology}, volume~2 of {\em Clay
  Mathematics Monographs}.
\newblock American Mathematical Society, Providence, RI; Clay Mathematics
  Institute, Cambridge, MA, 2006.

\bibitem[Nas52]{Nash1952Algebraic}
J.F. Nash.
\newblock Algebraic approximations of manifolds.
\newblock In {\em Proceedings of the International Congress of Mathematicians,
  Cambridge, Massachusetts, August 30--September 6, 1950}, volume~1, pages
  516--517, Providence, RI, 1952. American Mathematical Society.

\bibitem[OSS11]{OSS}
C.~Okonek, M.~Schneider, and H.~Spindler.
\newblock {\em Vector bundles on complex projective spaces}.
\newblock Modern Birkh\"{a}user Classics. Birkh\"{a}user/Springer Basel AG,
  Basel, 1988 edition, 2011.
\newblock With an appendix by S. I. Gelfand.

\bibitem[Poi04]{PoinareDefinitions}
H.~Poincar{\'e}.
\newblock Les d{\'e}finitions g{\'e}n{\'e}rales en math{\'e}matiques.
\newblock {\em Enseign. Math.}, 6:257--283, 1904.

\bibitem[Poi10]{PoincareTrans}
H.~Poincar\'{e}.
\newblock {\em Papers on topology}, volume~37 of {\em History of Mathematics}.
\newblock American Mathematical Society, Providence, RI; London Mathematical
  Society, London, 2010.
\newblock {\em Analysis situs} and its five supplements, Translated and with an
  introduction by John Stillwell.

\bibitem[Ram65]{Ramspeott1965Schnitte}
K.J. Ramspott.
\newblock Stetige und holomorphe schnitte in b{\"u}ndeln mit homogener faser.
\newblock {\em Mathematische Zeitschrift}, 89:234--246, 1965.

\bibitem[Rec66]{Rector1966}
D.L. Rector.
\newblock An unstable adams spectral sequence.
\newblock {\em Topology}, 5(4):343--346, 1966.

\bibitem[Rem56]{Remmert1956CRAS}
R.~Remmert.
\newblock Sur les espaces analytiques holomorphiquement s{\'e}parables et
  holomorphiquement convexes.
\newblock {\em Comptes Rendus Hebdomadaires des S{\'e}ances de l'Acad{\'e}mie
  des Sciences}, 243:118--121, 1956.

\bibitem[Rie04]{Riemann}
B.~Riemann.
\newblock {\em Collected papers}.
\newblock Kendrick Press, Heber City, UT, 2004.
\newblock Translated from the 1892 German edition by Roger Baker, Charles
  Christenson and Henry Orde.

\bibitem[Roq18]{roquette2018riemann}
P.~Roquette.
\newblock {\em The Riemann Hypothesis in Characteristic $p$ in Historical
  Perspective}, volume 2222 of {\em Lecture Notes in Mathematics}.
\newblock Springer, Cham, Switzerland, 2018.
\newblock History of Mathematics Subseries; bibliography and index included.

\bibitem[Ros96]{Rost1996ChowGroupsCoefficients}
M.~Rost.
\newblock Chow groups with coefficients.
\newblock {\em Documenta Mathematica}, 1:319--393, 1996.

\bibitem[Sch61a]{Schwarzenberger1961VectorBundles}
R.~L.~E. Schwarzenberger.
\newblock Vector bundles on algebraic surfaces.
\newblock {\em Proceedings of the London Mathematical Society}, 11(1):601--622,
  1961.

\bibitem[Sch61b]{Schwarzenberger1961VectorBundlesP2}
R.~L.~E. Schwarzenberger.
\newblock Vector bundles on the projective plane.
\newblock {\em Proceedings of the London Mathematical Society}, 11(1):623--640,
  1961.

\bibitem[Sch99]{Scholz}
E.~Scholz.
\newblock The concept of manifold, 1850--1950.
\newblock In {\em History of topology}, pages 25--64. North-Holland, Amsterdam,
  1999.

\bibitem[Ser54]{SerreFSBBKI}
J.-P. Serre.
\newblock Espaces fibr\'es alg\'ebriques.
\newblock In {\em S\'eminaire Bourbaki : ann\'ees 1951/52 - 1952/53 - 1953/54,
  expos\'es 50-100}, number~2 in S\'eminaire Bourbaki, pages 305--311.
  Soci\'et\'e math\'ematique de France, 1954.
\newblock talk:82.

\bibitem[Ser55]{SerreFAC}
J.-P. Serre.
\newblock Faisceaux alg\'{e}briques coh\'{e}rents.
\newblock {\em Ann. of Math. (2)}, 61:197--278, 1955.

\bibitem[Ser58]{SerreSplitting}
J.-P. Serre.
\newblock Modules projectifs et espaces fibr\'{e}s \`a fibre vectorielle.
\newblock In {\em S\'{e}minaire {P}. {D}ubreil, {M}.-{L}. {D}ubreil-{J}acotin
  et {C}. {P}isot, 1957/58, {F}asc. 2}, pages Expos\'{e} 23, 18.
  Secr\'{e}tariat math\'{e}matique, 11 rue Pierre Curie, Paris, 1958.

\bibitem[Ser86]{Serre1986OeuvresI}
J.-P. Serre.
\newblock {\em {\OE}uvres / Collected Papers. Volume I: 1949--1959}.
\newblock Springer Collected Works in Mathematics. Springer, Berlin,
  Heidelberg, 1986.

\bibitem[Ser91]{Serre1991PetitsCousins}
J.-P. Serre.
\newblock Les petits cousins.
\newblock In Peter Hilton, Friedrich Hirzebruch, and Reinhold Remmert, editors,
  {\em Miscellanea Mathematica}, pages 277--291. Springer-Verlag, Berlin, 1991.

\bibitem[Ser56]{SerreGAGA}
J.-P. Serre.
\newblock G\'{e}om\'{e}trie alg\'{e}brique et g\'{e}om\'{e}trie analytique.
\newblock {\em Ann. Inst. Fourier (Grenoble)}, 6:1--42, 1955/56.

\bibitem[Ste51]{Stein1951}
K.~Stein.
\newblock Analytische funktionen mehrerer komplexer veränderlichen zu
  vorgegebenen periodizitätsmoduln und das zweite cousinsche problem.
\newblock {\em Mathematische Annalen}, 123(1):201--222, 1951.

\bibitem[Ste99]{SteenrodFB}
N.~Steenrod.
\newblock {\em The topology of fibre bundles}.
\newblock Princeton Landmarks in Mathematics. Princeton University Press,
  Princeton, NJ, 1999.
\newblock Reprint of the 1957 edition, Princeton Paperbacks.

\bibitem[Sul04]{Sullivan2004GeometricTopology}
D.~Sullivan.
\newblock {\em Geometric Topology, Part I: Localization, Periodicity, and
  Galois Symmetry}.
\newblock K-Monographs in Mathematics. Kluwer Academic Publishers, Dordrecht,
  2004.

\bibitem[Sus00]{Suslin2011HigherChow}
A.A. Suslin.
\newblock Higher chow groups and etale cohomology.
\newblock In Eric~M. Friedlander, Andrei Suslin, and Vladimir Voevodsky,
  editors, {\em Cycles, Transfers, and Motivic Homology Theories}, volume 143
  of {\em Annals of Mathematics Studies}, pages 239--254. Princeton University
  Press, Princeton, NJ, 2000.

\bibitem[SV96]{SuslinVoevodsky1996}
A.~Suslin and V.~Voevodsky.
\newblock Singular homology of abstract algebraic varieties.
\newblock {\em Inventiones Mathematicae}, 123(1):61--94, 1996.

\bibitem[Voe98]{VICM}
V.~Voevodsky.
\newblock {${\mathbf A}^1$}-homotopy theory.
\newblock In {\em Proceedings of the {I}nternational {C}ongress of
  {M}athematicians, {V}ol. {I} ({B}erlin, 1998)}, number Extra Vol. I, pages
  579--604, 1998.

\bibitem[Voe02]{Voevodsky2002HigherChow}
V.~Voevodsky.
\newblock Motivic cohomology groups are isomorphic to higher chow groups in any
  characteristic.
\newblock {\em International Mathematics Research Notices}, 2002(7):351--355,
  2002.

\bibitem[Voe10]{VMEM}
V.~Voevodsky.
\newblock Motivic {E}ilenberg-{M}aclane spaces.
\newblock {\em Publ. Math. Inst. Hautes \'{E}tudes Sci.}, (112):1--99, 2010.

\bibitem[Wal35]{Walker1935ReductionSingularities}
R.~J. Walker.
\newblock Reduction of the singularities of an algebraic surface.
\newblock {\em Annals of Mathematics}, 36(2):336--365, April 1935.

\bibitem[Wei85]{Weierstrass1885AnalytischeDarstellbarkeit}
K.~Weierstrass.
\newblock {\"U}ber die analytische darstellbarkeit sogenannter
  willk{\"u}rlicher functionen einer reellen ver{\"a}nderlichen.
\newblock {\em Sitzungsberichte der K{\"o}niglich Preu{\ss}ischen Akademie der
  Wissenschaften zu Berlin}, pages 633--639, 1885.
\newblock Part I.

\bibitem[Wei95]{Weierstrass1895Continuous}
K.~Weierstrass.
\newblock {\"U}ber continuirliche functionen eines reellen arguments, die
  f{\"u}r keinen werth des letzteren einen bestimmten differentialquotienten
  besitzen.
\newblock In {\em Mathematische Werke}, volume~2, pages 71--74. Mayer \&
  M{\"u}ller, Berlin, 1895.

\bibitem[Wei48]{Weil1948}
A.~Weil.
\newblock {\em Vari{\'e}t{\'e}s ab{\'e}liennes et courbes alg{\'e}briques},
  volume 1064 of {\em Actualit{\'e}s Scientifiques et Industrielles}.
\newblock Hermann, Paris, 1948.
\newblock Publications de l'Institut de Math{\'e}matique de l'Universit{\'e} de
  Strasbourg, VIII.

\bibitem[Wei55]{WeilFBCourse}
A.~Weil.
\newblock Fibre spaces in algebraic geometry.
\newblock Notes by {A}. {Wallace}. {Department} of {Mathematics}, {University}
  of {Chicago} 48 p. (1955)., 1955.

\bibitem[Wei56]{Weil1956Abstract}
A.~Weil.
\newblock Abstract versus classical algebraic geometry.
\newblock In {\em Proceedings of the International Congress of Mathematicians,
  1954}, volume~3, pages 550--558, Amsterdam, 1956. North-Holland.

\bibitem[Wei62]{WeilFoundations}
A.~Weil.
\newblock {\em Foundations of algebraic geometry}.
\newblock American Mathematical Society, Providence, RI, 1962.

\bibitem[Wei79]{WeilI}
A.~Weil.
\newblock {\em Scientific works. {C}ollected papers. {V}ol. {I} (1926--1951)}.
\newblock Springer-Verlag, New York-Heidelberg, 1979.

\bibitem[Wei89]{Weibel1989}
C.A. Weibel.
\newblock Homotopy algebraic {$K$}-theory.
\newblock {\em Contemporary Mathematics}, 83:461--488, 1989.

\bibitem[Wei14]{WeilII}
A.~Weil.
\newblock {\em Oeuvres scientifiques/collected papers. {II}. 1951--1964}.
\newblock Springer Collected Works in Mathematics. Springer, Heidelberg, 2014.
\newblock Reprint of the 2009 [ MR2883739] and 1979 [ MR0537935] editions.

\bibitem[Whi35]{Whitneyspherespace}
H.~Whitney.
\newblock Sphere-spaces.
\newblock Proc. {Acad}. {USA} 21, 464-468 (1935)., 1935.

\bibitem[Whi41]{WhitneyFB}
H.~Whitney.
\newblock On the topology of differentiable manifolds.
\newblock In {\em Lectures in {T}opology}, pages 101--141. Univ. Michigan
  Press, Ann Arbor, MI, 1941.

\bibitem[Woo66]{Wood1966BanachBott}
R.~M.~W. Wood.
\newblock Banach algebras and bott periodicity.
\newblock {\em Topology}, 4:371--389, 1966.

\bibitem[Woo68]{Wood1968Polynomial}
R.~M.~W. Wood.
\newblock Polynomial maps from spheres to spheres.
\newblock {\em Inventiones Mathematicae}, 5:163--168, 1968.

\bibitem[Woo93]{Wood1993Quadrics}
R.~M.~W. Wood.
\newblock Polynomial maps of affine quadrics.
\newblock {\em Bulletin of the London Mathematical Society}, 25(5):491--497,
  1993.

\bibitem[WW20]{WickelgrenWilliamsSurvey}
K.~Wickelgren and B.~Williams.
\newblock Unstable motivic homotopy theory.
\newblock In {\em Handbook of homotopy theory}, CRC Press/Chapman Hall Handb.
  Math. Ser., pages 931--972. CRC Press, Boca Raton, FL, [2020] \copyright
  2020.

\bibitem[Zar39]{Zariski1939ReductionSingularities}
O.~Zariski.
\newblock The reduction of the singularities of an algebraic surface.
\newblock {\em Annals of Mathematics}, 40(3):639--689, July 1939.

\bibitem[Zar42]{Zariski1942SimplifiedResolution}
O.~Zariski.
\newblock A simplified proof for the resolution of singularities of an
  algebraic surface.
\newblock {\em Annals of Mathematics}, 43(3):583--593, July 1942.

\end{thebibliography}
\end{footnotesize}
\Addresses
\end{document}